\documentclass{article}
\usepackage[a4paper]{geometry}
\usepackage[english]{babel}
\usepackage[utf8]{inputenc}
\usepackage[toc,page]{appendix}
\usepackage{amsthm}
\usepackage[small]{titlesec}
\usepackage{mathtools}
\usepackage{graphicx}
\usepackage{amsfonts}
\usepackage{amssymb}
\usepackage{mathrsfs}
\usepackage{xcolor}
\usepackage{esint}
\usepackage{relsize}
\usepackage{comment}
\usepackage{bigints}
\usepackage{dsfont}
\usepackage{geometry}
\usepackage{float}
\usepackage{bbm}
\usepackage{BOONDOX-cal}

\numberwithin{equation}{section}
\theoremstyle{plain}
\newtheorem{thm}{Theorem}[section]
\newtheorem{cor}[thm]{Corollary}
\newtheorem{lem}[thm]{Lemma}
\newtheorem{prop}[thm]{Proposition}
\theoremstyle{definition}
\newtheorem{defn}{Definition}
\newtheorem{remark}{Remark}
\newtheorem{assumption}{Assumption}
\theoremstyle{remark}

\newcommand{\R}{\mathbb{R}^d}
\newcommand{\Ll}{\mathscr{L}}
\newcommand{\M}{\mathscr{M}_{+}\left(\R \right)}

\newcommand{\vp}{\varphi}

\newcommand{\D}{\mathscr{D}}

\newcommand*\dd{\mathop{}\!\mathrm{d}} 

\title{The global attractor of the inelastic linear Boltzmann equation in a gravity field for Maxwell molecules}

\author{Th\'eophile Dolmaire, Nicola Miele, Alessia Nota}

\date{\today }

\begin{document}

\maketitle

\begin{abstract}
\noindent
In this article we consider the linear inelastic Boltzmann equation in presence of a uniform and fixed gravity field, in the case of Maxwell molecules.   
We first obtain a well-posedness result in the space of finite, non-negative Radon measures. In addition, we rigorously prove 
the existence of a stationary solution under the non-equilibrium condition which is induced by the presence of the external field. We further show that this stationary solution is unique in the class of the finite, non-negative Radon measures with finite first order moment, and that all the solutions in this class converge towards the stationary solution in the weak topology of the measures.  
\end{abstract}

\tableofcontents

\maketitle

\section{Introduction} 

\subsection{The model}

In the present article, we will consider an inelastic Lorentz gas with Maxwell molecule interactions under the action of a uniform and fixed acceleration field, in dimension $d =2$ or $d=3$. The evolution of the tagged particle interacting inelastically with a set of fixed and randomly distributed scatterers is described by the inelastic linear Boltzmann equation, with collision kernel of the form:
\begin{align}
\label{EQUATIntroMaxwlColliKernl}
B\left(\left\vert N \cdot \omega \right\vert,\vert v \vert\right) = b \left(\left\vert N\cdot \omega \right\vert \right) \hspace{3mm} \text{with} \hspace{3mm} \quad N:=\frac{v}{\vert v \vert},
\end{align}
which is the Maxwell molecules collision kernel. The presence of the uniform acceleration field is expressed by the presence of a drift term in the equation. We will consider the space homogeneous situation, that is, the case when the distribution function $f$ of the tagged particle is independent from the position variable, and depends only on the time variable $t$, and the velocity $v$ of the tagged particle. We will therefore consider the equation:
\begin{align}
\label{EQUATIntroLineaInelaBoltz}
\partial_t f(t,v) + a \cdot \partial_v f(t,v) = \int_{\mathbb{S}^{d-1}} \frac{1}{r} b(| \hspace{0.5mm}'\hspace{-0.5mm}N \cdot \omega|)f(t,'\hspace{-1mm}v) \dd \omega - \left( \int_{\mathbb{S}^{d-1}} b(|N \cdot \omega|) \dd \omega \right) f(t,v),
\end{align}
where we used the notation $\hspace{0.5mm}'\hspace{-0.5mm}N = \hspace{0mm}'\hspace{-0.5mm}v/\vert '\hspace{-0.5mm}v \vert$, $a \in \mathbb{R}^d$ is the acceleration field, and $'\hspace{-0.5mm}v$ is the pre-collisional velocity of $v$, that is, such that a particle colliding with a scatterer with velocity
$'\hspace{-0.5mm}v$ and angular parameter $\omega$ is reflected with the velocity $v$. Finally, we assume that the inelastic collisions take place according to the following collision law:
\begin{align}
v =\hspace{0.0mm} '\hspace{-0.5mm}v - (1+r)('\hspace{-0.5mm}v\cdot\omega)\omega \hspace{3mm} \text{or equivalently} \hspace{3mm} '\hspace{-0.5mm}v = v - (1+1/r)(v\cdot\omega)\omega,
\end{align}
where $r \in \, \, ]0,1]$ is the \emph{restitution coefficient}, which is assumed through all the paper to be a fixed real number (the case $r=1$ corresponding to the elastic case).\\
\newline
\noindent
We will also consider the following generalization of \eqref{EQUATIntroLineaInelaBoltz}: the inelastic Boltzmann equation with Maxwell molecules interacting with fixed scatterers, and \emph{rehomogeneized collision operator}. Notice that this procedure will lead to a nonlinear version of the collision operator, satisfying however a physically relevant decay of the temperature. Eventually, the collision operator will be:
\begin{align}
\label{EQUATIntroLineaInelaBoltzRehom}
\partial_t f(t,v) &+ a \cdot \partial_v f(t,v)  \nonumber\\
&=
T^\mu(t)\int_{\mathbb{S}^{d-1}} \frac{1}{r}  b(|\hspace{0.5mm}'\hspace{-0.75mm}N \cdot \omega|) f(t,'\hspace{-1mm}v) \dd \omega - T^\mu(t)\left( \int_{\mathbb{S}^{d-1}}  b(|N \cdot \omega|) \dd \omega \right) f(t,v) 
\end{align}
with
\begin{equation}
T(t) = \frac{1}{2}\int_{\mathbb{R}^d} \vert v \vert^2 f(t,v) \dd v \hspace{2mm} \text{and} \hspace{2mm} \mu \geq 0.
\end{equation}
We discuss the well-posedness theory and the long time behaviour of \eqref{EQUATIntroLineaInelaBoltz}. More precisely, we prove that \eqref{EQUATIntroLineaInelaBoltz} is well-posed, globally in positive times, in the space of positive, finite Radon measures on $\mathbb{R}^d$, and if the initial data of \eqref{EQUATIntroLineaInelaBoltz} have finite first and second order moments, then so are also the first and second order moments of the associated solutions.\\
In addition, we prove that there exists a unique steady state $f_\infty$ to the equation \eqref{EQUATIntroLineaInelaBoltz} in the class $\mathcal{M}_+(\mathbb{R}^d)$ of positive, finite Radon measures, and that such a steady state is the unique global attractor in the class of measures of $\mathcal{M}_+(\mathbb{R}^d)$ with finite first order moment. In other words, for any initial datum $f_0$ of $\mathcal{M}_+(\mathbb{R}^d)$ with finite first order moments, the solution $f$ converges towards $f_\infty$, in the sense of weak convergence of the measures.\\
Finally, we provide a complete study of the system of the moments of zeroth, first and second order of \eqref{EQUATIntroLineaInelaBoltzRehom}, proving that the first and second order moments converge towards unique limiting values, for any physical initial data.

\subsection{Lorentz gas and granular material: a brief review of the literature}

\paragraph{Existing results concerning conservative Lorentz gases.} The Lorentz gas model, introduced by H. A. Lorentz in 1905 to study electron transport in metals, describes the dynamics of non-interacting particles moving through a fixed random configuration of heavy scatterers. The interaction between the Lorentz particle and the scatterers is characterized by a finite-range central potential.  Despite its simplicity, it represents a yet profoundly significant model providing key insights on how microscopic reversibility can be reconciled with macroscopic irreversibility.  In fact, from this model, one can obtain, under suitable
scaling limits, linear kinetic equations (\cite{Gallavotti, BBS, LuTo, N, NSV, Spohn1, Spohn2}), and, from these, one can derive diffusion equations in longer time scales (cf.  \cite{BGSR016, BaNP014, BNP, LuTo}).\\
In most of the mathematical studies of the linear Boltzmann equation, it has been assumed that additional transport terms describing the effect of possible external fields are absent.  
However, the presence of external fields significantly affects its derivation in low-density regimes, as well as the properties of its solutions.  More precisely, in \cite{BMHH} and later in \cite{BHPH} (see also \cite{KuSp} where the model has been studied numerically), it has been shown that the motion of a test particle in $\mathbb{R}^2$ with a Poisson distribution of hard disk scatterers and a uniform, constant magnetic field perpendicular to the plane formally leads to the generalized Boltzmann equation with
memory terms. A rigorous derivation of this equation has been recently obtained in \cite{NSS}. We also refer to \cite{MN},
where linear kinetic equations with a magnetic transport term have been derived.
On the other hand we also mention \cite{BGG} where the well-posedness of a slightly different linear Boltzmann equation  with external field, specifically a space periodic electric field, has been studied in the time-stationary regime.\\
\noindent
It is important to note that all the results mentioned above pertain to systems where the dynamics of the Lorentz particle is conservative, implying that each collision with the background obstacles is perfectly elastic. In contrast, issues concerning the well-posedness of solutions to linear kinetic equations in the presence of external fields, as well as the long-term behavior and thermalization of solutions, have been significantly less explored in non-conservative systems. 

\paragraph{Qualitative behaviour of granular gases. } Turning to the case of dissipative gases, we first mention that considering inelastic collisions between the elementary components of a fluid is motivated by many physical situation: snow, sand or interstellar dust may be modelled by a very large number of particles that dissipate kinetic energy during collisions. Such a model has countless applications, in theoretical physics as well as in industry (see in particular \cite{BeJN996}). Systems composed of a large number of particles that interact inelastically, the so-called \emph{granular media}, present a wide range of very interesting behaviours, and exhibit differences of fundamental nature with respect to conservative fluids, the evolution of which is governed in the low density limit by the classical, elastic Boltzmann equation. For a brief introduction to the peculiarities of the granular media, the reader may consult the surveys \cite{BeJN996} and \cite{Kada999}. Let us present two of such specificities, which are maybe the most prominent.\\
Within granular media, the dissipation of kinetic energy during collisions induces a decay of the temperature in the case when there is no external source that is injecting energy in the system. In many cases, it is possible to prove that the temperature decays according to a power law: this is the celebrated Haff's law \cite{Haff983}. Observe that the algebraic exponent of such a decay depends on the properties of the restitution coefficient (see \cite{BrPo004}).\\
A second characteristic of granular media is the trend to create spontaneously stable clusters (or regular patterns in some particular cases), even when the gas departs from an initial state very close to be spatially homogeneous. This behaviour is widely documented (\cite{GoZa993}, \cite{Kada999}, \cite{MNYo996}, \cite{BrPo004}), but still not completely understood.\\
In addition, in the case of a fixed restitution coefficient that is small enough, inelastic particle systems exhibit \emph{inelastic collapse}: infinitely many collisions take place in finite time. At the level of the kinetic equation, this phenomenon can be identified by the convergence in finite time of the solutions towards Dirac masses, which represents a brutal concentration of the velocity profile of the distribution function. A vast literature exists about this phenomenon: see \cite{MNYo994}, \cite{BeCa999}, \cite{BrPo004}, \cite{PoSc005}, \cite{DoVe024}, \cite{DoHu024}, \cite{DoVe025}, \cite{Dolm025} and the references therein.\\
On the one hand, the inelastic collapse phenomenon can be seen as an extreme case of clusterization. Such a phenomenon constitutes a major obstruction to perform numerical simulations, and to provide a rigorous derivation of inelastic kinetic equations (no analog of Lanford's theorem \cite{Lanf975} exists for such models). On the other hand, it is still unclear if such a phenomenon exists for models different from the particular case of fixed restitution coefficient, nor if the dynamics of the particle system is globally well-posed, even when a collapse occurs (see \cite{DoVe024}, \cite{DoVe025}).

\subsection{Mathematical theory of the inelastic Boltzmann equation}
\label{SSectIntroMatheTheorInelaBoltz}

In this section, we briefly review the state of the art concerning the inelastic Boltzmann equation. Contrary to the elastic Boltzmann equation, the dissipative character of the collisions causes the solutions to develop singularities (concentrations of the velocity profile, explosion of the gradient $\partial_x f$...), and the trend to create clusters prevents to consider linearization around an homogeneous solution. For these reasons, the mathematical theory of the inelastic Boltzmann equation is still quite incomplete, and many results are obtained only in particular cases, or for related, simpler models.

\paragraph{The non-linear inelastic Boltzmann equation.} We refer to \cite{Vill006} for a very complete survey about the mathematical state of the art concerning the non-linear version of the equation, and to the more recent \cite{CHMR021}, which addresses also the question of the numerical simulations.\\
Rigorous results were first obtained in the one-dimensional case (\cite{BeCP997}, \cite{BeCP997_2}, \cite{Tosc000}), before that the case of higher dimension was studied by many authors. This research was motivated by from a conjecture due to Ernst and Brito \cite{ErBr003}, postulating that, while the temperature of a granular gas decays according to the Haff's law, the velocity profile $f(t,v)$ should converge towards a self-similar solution of the inelastic Boltzmann equation:
\begin{align}
f(t,v) \xrightarrow[t\rightarrow+\infty]{} \alpha(t) m(\beta(t) \vert v \vert),
\end{align}
where $\alpha,\beta:\mathbb{R}_+ \rightarrow \mathbb{R}_+^*$ are two scaling functions, and such that the self-similar profile $m$ should present overpopulated high energy tails: $m(\vert v \vert) \sim e^{-a \vert v \vert^b}$ as $\vert v \vert \rightarrow +\infty$ for some $0 < b < 2$ and $a > 0$.\\
At the same time, a simplified version of the inelastic Boltzmann equation was introduced, corresponding to Maxwell molecule interaction (the assumption \eqref{EQUATIntroMaxwlColliKernl} on the collision kernel). In such a case, the evolution of the moments can be determined. Nevertheless, as observed for instance in \cite{Vill006}, such a model would lead to an exponential decay of the temperature, which does not correspond to the algebraic decay prescribed by the Haff's law. For this reason, several authors (\cite{CaCG000}, \cite{BoCe003}) considered a rehomogeneization factor $T^\mu(t)$ in front of the collision operator (as in \eqref{EQUATIntroLineaInelaBoltzRehom}) in order to restore  the homogeneity of the collision operator, and to recover an ad hoc Haff's law.\\
In order to restore also the existence of steady states to the inelastic Boltzmann equation, variations of the space-homogeneous version for Maxwell molecules were considered. For instance, it is possible to consider the inelastic particles in a thermal bath, reinjecting the energy that is lost during the collisions (first in \cite{BCCP998}, then, for instance, in \cite{CaCG000}). Another effect competing with the inelastic collisions that was considered is the presence of a shear \cite{Cerc001}.\\
Finally, important progresses were completed in the series of articles \cite{MMRR006}, \cite{MiMo006} and \cite{MiMo009}, in which the space-homogeneous, inelastic Boltzmann equation is considered for the hard sphere collision kernel, which corresponds better to a physical model than the Maxwell molecule kernel. In particular, the Cauchy theory is established in a quite general setting, and the Haff's law as well as the Ernst-Brito conjecture are proved for restitution coefficients close enough to $1$. We mention also \cite{AlLo010}, in which the Haff's law is proved for viscoelastic particles, with optimal lower and upper bounds.\\
To conclude this discussion concerning the non-linear inelastic Boltzmann equation, we emphasize that, to the best of our knowledge, the only references that deal with the spatially non-homogeneous version of (in general dimension) are \cite{Alon009} and \cite{AlLT024}, where the well-posedness of the solutions is obtained only close to the vacuum in \cite{Alon009}, and close to thermal equilibrium, for restitution coefficients $r$ close enough to $1$ (that is, close enough to the elastic case) in \cite{AlLT024}.

\paragraph{The linear inelastic Boltzmann equation.} 
Linear inelastic models have been less discussed in the literature. Concerning the Rayleigh gas, that is, the evolution of a tagged particle within a background gas at thermal equilibrium, several directions were investigated. The existence of solutions, as well as long-time behaviour considerations are discussed in \cite{Pett004}. The discussion of the existence of an equilibrium, and the convergence of the solutions towards such an equilibrium can be found in \cite{SpTo004} in the case of Maxwell molecules and in \cite{LoTo004} in the case of hard spheres. In both cases, it is proved that the solutions converge towards a Maxwellian steady state (and so, contrary to the non-linear case, the tail of such a profile does not satisfy the Ernst-Brito conjecture).\\
\newline
\noindent
In the case of a Lorentz gas, that is, when the tagged particles collide with fixed obstacles, there is no source of energy to compensate the inelastic collisions, and so the gas cools down. The behaviour of such a system was studied in \cite{HDHa001}, relying on formal and numerical investigations. In particular, the diffusion of the particles in space is described, as well as the evolution of the velocity profile.\\
In order to have the existence of non-trivial steady states, the inelastic Lorentz gas was also considered when evolving under the action of an acceleration field (that can be seen as the gravity, or as an electric or magnetic field acting on charged particles). Such a situation was studied in \cite{MaPi999}, where a Lorentz gas of inelastic hard spheres evolves under the action of a constant acceleration field. In this reference, the steady states of the linear inelastic Boltzmann equation are completely determined in dimension $d=1$, and a series expansion of a steady state is given in dimension $d=3$, in the limit of small inelasticity.\\
In \cite{MaPi007}, the case of an inelastic Lorentz gas composed with Maxwell particles is considered. The existence of a steady state is obtained, based on the use of the Fourier transform. In particular, an explicit, integral formula is given for the steady state. Then, the formula of such a steady state is considered in the limit of small inelasticity.\\
We observe that in \cite{MaPi999} and \cite{MaPi007}, the uniqueness of the steady states and their stability are not discussed. In addition, the explicit formulas are good approximations only in the case when the particles interact almost elastically with the scatterers.

\subsection{Notations}

For the rest of the article, we will make use of the following notations. We will work in the space of continuous functions on $\mathbb{R}^d$. Depending on the situation, we will consider real- or complex-valued continuous functions, belonging to spaces respectively denoted by $\mathcal{C}(\mathbb{R}^d,\mathbb{R})$ and $\mathcal{C}(\mathbb{R}^d,\mathbb{C})$. When the context will make it clear, we will omit the space in which the functions take their values, so that $\mathcal{C}(\mathbb{R}^d)$ stands for $\mathcal{C}(\mathbb{R}^d,\mathbb{C})$ or $\mathcal{C}(\mathbb{R}^d,\mathbb{C})$. Nevertheless, we will also consider non-negative real valued functions of $\mathcal{C}(\mathbb{R}^d)$, the space of such functions will be denoted by $\mathcal{C}(\mathbb{R}^d,\mathbb{R}_+)$.\\
We will also denote by:
\begin{itemize}
\item $\mathcal{C}^\infty(\mathbb{R}^d)$ the space of infinitely differentiable functions on $\mathbb{R}^d$, and $\mathcal{C}_c^\infty(\mathbb{R}^d)$ its subspace of compactly supported functions,
\item $\mathcal{C}_0(\mathbb{R}^d)$ the space of continuous functions vanishing at infinity,
\item $\mathcal{C}_b(\mathbb{R}^d)$ the space of continuous, bounded functions,
\item these spaces will be equipped with the uniform norm, denoted by $\vert \hspace{-0.5mm} \vert \cdot \vert \hspace{-0.5mm} \vert_\infty$.
\end{itemize}
In addition, we introduce the following norms:
\begin{itemize}
\item for any $p > 0$:
\begin{align*}
\vert \hspace{-0.5mm} \vert \cdot \vert \hspace{-0.5mm} \vert_{\infty,p} = \sup_{\xi} \frac{\vert \cdot \vert}{1 + \vert \xi \vert^p},
\end{align*}
and the space $\mathcal{C}_p(\mathbb{R}^d,\mathbb{C})$, denoted also by $\mathcal{C}_p$ in short, as the subspace of $\mathcal{C}(\mathbb{R}^d)$ with a finite $\vert \hspace{-0.5mm} \vert \cdot \vert \hspace{-0.5mm} \vert_{\infty,p}$ norm,
\item the supremum norm in $t$ in the interval $I \subset \mathbb{R}$ for the $\vert \hspace{-0.5mm} \vert \cdot \vert \hspace{-0.5mm} \vert_{\infty,p}$ norm: $\vert \hspace{-0.5mm} \vert \hspace{-0.5mm} \vert \cdot \vert \hspace{-0.5mm} \vert \hspace{-0.5mm} \vert_I:=\sup_{t\in I}\vert \hspace{-0.5mm} \vert \cdot \vert \hspace{-0.5mm} \vert_{\infty,p}$,
\item the total variation norm: $\|\cdot\|_{\mathscr{M}}$,
\item the operator norm:  $\|\cdot\|$ or $\|\cdot\|_X$ in case we want to emphasize the Banach space $X$.
\end{itemize}
\noindent
Moreover, we will denote by $\mathscr{M}(\R)$ the space of finite signed Radon measures on $\R$ and by $\M$ the non-negative cone of $\mathscr{M}(\R)$. For a general locally compact topological space $X$ it is possible to characterize the space by duality. More precisely, the classical Riesz-Markov-Kakutani theorem, see for instance \cite{Rudin},  provides
\begin{equation}
        (\mathcal{C}_0(X))^*=\mathscr{M}(X).
\end{equation}
We recall the usual embeddings of the functional spaces that we will constantly use throughout this article:
\begin{align}
\mathcal{C}^\infty_c(\mathbb{R}^d) \subseteq \overline{\mathcal{C}^\infty_c(\mathbb{R}^d)}^{\vert\hspace{-0.5mm}\vert \cdot \vert\hspace{-0.5mm}\vert_\infty} = \mathcal{C}_0(\mathbb{R}^d) \subseteq \mathcal{C}_b(\mathbb{R}^d).
\end{align}
For a measure $f \in\mathscr{M}(X)$ we denote its total variation norm, or equivalently its norm as a continuous linear functional on $\mathcal{C}_0(X)$,  by $\|f\|_{\mathscr{M}}$. Moreover we notice that since the measure is finite, the equality $\|f\|_{\mathscr{M}}=|f|(X)$ is well-defined. We will use the weak$-*$ topology on the space $\mathcal{C}_0(X)$, and we will refer to this topology as the \emph{weak topology of $\mathscr{M}(X)$}, as customary in the literature (e.g., see \cite{Fe2}).

\subsection{Main results}

Our starting point is the Cauchy problem for \eqref{EQUATIntroLineaInelaBoltz}, i.e.
\begin{equation}\label{eq:LinInBolstrong_Cauchy}
    \begin{cases}
 \displaystyle\partial_t f(t,v) + a \cdot \partial_v f(t,v) = \int_{\mathbb{S}^{d-1}} \frac{b( \vert \hspace{0.5mm}'\hspace{-0.75mm}N\cdot \omega\vert )}{r}  f(t,'\hspace{-1mm}v) \dd \omega - \left( \int_{\mathbb{S}^{d-1}}  b( \vert N\cdot \omega \vert ) \dd \omega \right) f(t,v),
& \\
        f(0,v)=f_0(v),
    \end{cases}
\end{equation}
where
\begin{align}
\label{EQUAT_Loi_Colli}
N = \frac{v}{\vert v \vert},\hspace{5mm} \hspace{0.5mm}'\hspace{-0.75mm}N = \frac{\hspace{0.5mm}'\hspace{-0.5mm}v}{\vert '\hspace{-0.5mm}v \vert},\hspace{5mm} \text{and} \hspace{5mm} '\hspace{-0.5mm}v = v - (1+1/r) \big(v\cdot\omega\big)\omega.
\end{align}
The general assumption on the collision kernel $B$ we will use in this paper is the following

\begin{assumption}\label{ass:kernelB} 
We assume that the angular part $b$ of the  collision kernel $B$ satisfies 
\begin{equation}\label{ass:b}
    b(\vert N \cdot \omega\vert ) \in L^{\infty}([0,1]), \quad \text{where} \quad N:= \frac v {\vert v \vert }\, .
\end{equation}
\end{assumption}
\noindent
We now introduce the definition of weak solutions for the equation \eqref{eq:LinInBolstrong_Cauchy} that we will use.

\begin{defn}[Weak solution of \eqref{eq:LinInBolstrong_Cauchy}]\label{def:weakSol}
 	Let $f_0 \in \M$ be a non-negative finite Radon measure, let $T>0$ be a positive number and let the collision kernel $b(\vert N \cdot \omega\vert )$ satisfy Assumption \ref{ass:kernelB}. A measure $f \in \mathcal{C}\left( [0,+\infty[,\M\right)$ is a \emph{weak solution} of \eqref{eq:LinInBolstrong_Cauchy} if for every test function $\vp \in \mathcal{C}^1\left( [0,+\infty[,\mathcal{C}^1_0(\R) \right) $ one has
\begin{equation}\label{eq:weakf} 
\int_{\R}\vp(T,v)f(T,\dd v)- \int_{\R}\vp(0,v)f_0(\dd v) = \int_0^T \dd t \int_{\R} f(t,\dd v)\left[\partial_t\vp + a \cdot \partial_v \vp+\Ll^*(\vp) \right](t,v)\dd t,
\end{equation}
where
\begin{equation}\label{eq:adjCollOp}
    \Ll^*(\vp) = \int_{\mathbb{S}^{d-1}} b( \vert N\cdot \omega \vert) \left[\vp(v')-\vp(v)\right]\dd \omega
\end{equation}
is the \emph{adjoint} of the collision operator $\Ll$ and $v'$ is defined as $v' = v - (1+r)(v\cdot\omega)\omega$.
\end{defn}

\noindent
We first provide a well-posedness result for \eqref{eq:LinInBolstrong_Cauchy}, whose proof will be presented in Section \ref{sec:WP}.

\begin{thm}[Well-posedness]\label{thm:wp1}
Let $d = 2$ or $3$, and let $a \in \mathbb{R}^d$ be a fixed vector. 
Assume that the restitution coefficient $r$ satisfies $0 < r \leq 1$. Let $f_0\in \mathscr{M}_+(\mathbb{R}^d)$ and let $b(\vert N \cdot \omega\vert )$ satisfy Assumption \ref{ass:kernelB}. Then, there exists a unique weak solution $f \in \mathcal{C}\left( [0,+\infty),\M \right)$ of \eqref{eq:LinInBolstrong_Cauchy}  in the sense of Definition \ref{def:weakSol}.
\end{thm}

\begin{remark}
Observe that the case $r = 1$ is covered by Theorem \ref{thm:wp1}, which corresponds to the elastic case.
\end{remark}

\noindent
The main purpose of this paper is to prove the existence, uniqueness and stability of a stationary solution to \eqref{eq:LinInBolstrong_Cauchy}. To this end we introduce the following concept of \emph{stationary solution} to \eqref{eq:LinInBolstrong_Cauchy}.

\begin{defn}[Stationary solution to \eqref{eq:LinInBolstrong_Cauchy}] \label{def:stationarySol}
    Let $d = 2$ or $3$, and let $a \in \mathbb{R}^d$. Assume that the restitution coefficient $r$ satisfies $0 < r \leq 1$. Let $f_0\in \mathscr{M}_+(\mathbb{R}^d)$ and let $b(\vert N \cdot \omega\vert )$ satisfy Assumption \ref{ass:kernelB}. A measure $f \in \M$ is a stationary solution to \eqref{eq:LinInBolstrong_Cauchy} if  for every $\vp \in \mathcal{C}^1_0(\R)$ one has 
    \begin{equation}
        \int_{\R}f(\dd v)\left[a \cdot \partial_v \vp + \Ll^*(\vp)\right](v)=0
    \end{equation}
    where $\Ll^*$ is as in \eqref{eq:adjCollOp}. 
\end{defn}

\noindent
We have the following theorem whose proof can be found in Section \ref{ssec:uniqueness}.

\begin{thm}[Existence and uniqueness of steady states]   \label{thm:steadystate}
Let $d = 2$ or $3$, and let $a \in \mathbb{R}^d$. Assume that the restitution coefficient $r$ satisfies $0 < r < 1$ (we assume in particular that $r \neq 1$). Let $b(\vert N \cdot \omega\vert )$ satisfy Assumption \ref{ass:kernelB}. Then, there exists a unique stationary solution $f_\infty\in \mathscr{M}_+(\R)$ in the sense of Definition \ref{def:stationarySol}.
Moreover,  $\{|v|,|v|^2\}f_{\infty} \in \M$. 
\end{thm} 

\noindent
To prove the existence of a stationary solution we will employ an abstract fixed point argument, specifically the Schauder fixed point theorem. This theorem is a general and flexible tool that can be applied to various kinetic equations. However, it is less effective for establishing the uniqueness and stability of stationary solutions. Therefore, to prove the uniqueness of the steady state, we will utilize a more direct approach through the Neumann series expansion.  
Moreover, we can show that the stationary solution is asymptotically stable, which is the content of the following theorem, whose proof will be presented in Section \ref{sec:stability}.

\begin{thm}[Stability of steady states]\label{thm:stability}
Let $d = 2$ or $3$, and let $a \in \mathbb{R}^d$.  
We assume that the restitution coefficient $r$ satisfies $0 < r < 1$ (we assume in particular that $r \neq 1$) and that $b(\vert N \cdot \omega\vert )$ satisfy Assumption \ref{ass:kernelB}. Let $f_0(v)\in \mathscr{M}_+(\R)$ such that $|v|f_0 \in \M$. Let $f \in \mathcal{C}\left( [0,+\infty),\M \right)$ be the solution to \eqref{eq:LinInBolstrong_Cauchy} obtained in Theorem \ref{thm:wp1} and $f_\infty \in \M $ be the stationary solution to \eqref{eq:LinInBolstrong_Cauchy} obtained in Theorem \ref{thm:steadystate}. Then:
\begin{align}
f(t,v) \rightarrow f_\infty(v),\quad \text{as}\ \ t\to \infty,
\end{align}
in the weak topology of $\mathscr{M}_+(\R)$. 
\end{thm} 

\noindent
The main tool that we will use to prove Theorem \ref{thm:stability} is the well developed machinery available for the study of the Boltzmann equations in the case of Maxwell molecules by means of the Fourier transform method that was introduced by Bobylev in \cite{Bo75}. In the end, it is possible to prove the local uniform convergence of the Fourier transforms of the solutions towards the Fourier transform of the steady state. As a consequence, general results of probability theory (e.g., \cite{Fe2}) allow to state the weak convergence of the measure-valued solutions of \eqref{eq:LinInBolstrong_Cauchy}.

\begin{remark}
We notice that the approach used to obtain Theorem \ref{thm:stability} allows also to establish the well-posedness of the inelastic linear Boltzmann equation \eqref{eq:LinInBolstrong_Cauchy}, under an additional assumption on the first order moment. Nevertheless, the stability of the steady state is obtained under weaker assumptions as in Theorem \ref{thm:steadystate}. Therefore, Theorem \ref{thm:stability} establishes that if another steady state of \eqref{eq:LinInBolstrong_Cauchy} exists, then its first order moment cannot be finite. For this reason, we separate and rely on these two approaches.\\
Moreover, it might be that the semigroup approach can also be fruitfully applied to investigate the case of more general collision kernels.
\end{remark}

\noindent
Furthermore, we will study the system of moments associated to the rehomogeneized (nonlinear) equation \eqref{EQUATIntroLineaInelaBoltzRehom} for Maxwell collision kernel, which we recall here:
\begin{align}
\label{EQUATRepetEquatRehom}
\partial_t f(t,v) &+ a \cdot \partial_v f(t,v)  \nonumber\\
&=
T^\mu(t)\int_{\mathbb{S}^{d-1}} \frac{1}{r} b \left( \left\vert \hspace{0.5mm}'\hspace{-0.75mm} N \cdot \omega \right\vert \right) f(t,'\hspace{-1mm}v) \dd \omega - T^\mu(t)\left( \int_{\mathbb{S}^{d-1}} b \left( \left\vert N \cdot \omega \right\vert \right) \dd \omega \right) f(t,v) 
\end{align}
with $T(t) = \frac{1}{2} \int_{\mathbb{R}^d} \vert v \vert^2 f(t,v) \dd v $ and $\mu \geq 0.$ 
More precisely, we define the moments $M_0(t), \, M_1(t),\, M_2(t)$ associated to a solution $f$ of \eqref{EQUATRepetEquatRehom} as
\begin{align}
    M_0(t) & = \int_{\R} f(t,\dd v); \label{eq:M_0}\\
    M_1(t) & =\int_{\R} vf(t,\dd v); \label{eq:M_1}\\
    M_2(t) & =\int_{\R} |v|^2f(t,\dd v) \label{eq:M_2}.
\end{align}

\noindent
In particular, assuming $\mu$ to be non-negative, we can describe completely the long-time behaviour of the three first moments of the solution of the rehomogeneized inelastic Boltzmann equation \eqref{EQUATRepetEquatRehom}, assuming that such a solution exists, as well as its first moments. This is the content of the following proposition whose proof will be provided in Section \ref{sec:Generalmoments}. 

\begin{prop}[Long time-behaviour of the moments]
Let $d = 2$ or $3$, and let $a \in \mathbb{R}^d$. Assume that the restitution coefficient $r$ satisfies $0 < r < 1$ (we assume in particular that $r \neq 1$). Let $\mu \geq 0$ and let us assume that $f \in \mathcal{C}([0,+\infty[,\M)$ is a solution of \eqref{EQUATIntroLineaInelaBoltzRehom} in the sense of Definition \ref{def:weakSol} such that $\{|v|,|v|^2\}f(t) \in \M$ for any $t \geq 0$. Then:
\begin{align}
\frac{\dd}{\dd t} M_0(t) = 0, \hspace{3mm} \forall t \geq 0,
\end{align}
\begin{align}
M_1(t) \rightarrow M_{1,\infty}, \hspace{3mm} M_2(t) \geq 0 \hspace{3mm} \forall t \geq 0, \hspace{3mm} \text{and} \hspace{3mm} M_2(t) \rightarrow M_{2,\infty},
\end{align}
as $t\rightarrow +\infty$, with $M_{1,\infty} \in \mathbb{R}^d$ and $M_{2,\infty} \geq 0$, where  $M_0(t),\, M_1(t),\, M_2(t)$ have been defined in \eqref{eq:M_0}, \eqref{eq:M_1}, \eqref{eq:M_2} respectively.
\end{prop}

\section{Well-posedness results}\label{sec:WP}
To prove Theorem \ref{thm:wp1} we work with the adjoint equation of \eqref{eq:LinInBolstrong_Cauchy}, relying on the duality provided by the Riesz-Markov-Kakutani theorem. We notice that for cut-off kernels it is possible to decompose $\Ll^*$ as
\begin{equation}\label{eq:decAdjL}
\Ll^*(\vp)=\int_{\mathbb{S}^{d-1}}b \left(|N \cdot \omega|\right)\vp(v')\dd \omega-C_b \vp
\end{equation} 
with $C_b= \int_{\mathbb{S}^{d-1}} b(|N \cdot \omega|) \dd \omega.$ From the weak formulation \eqref{eq:weakf} we deduce the \emph{backward in time} adjoint of \eqref{eq:LinInBolstrong_Cauchy}, which reads  
\begin{equation}\label{eq:adjEq}
	-\partial_t  {\vp}-a \cdot \partial_v \vp = \Ll^*({\vp}),
 \end{equation}
 with final condition
 \begin{equation}\label{eq:adjEqT}
 \vp(T,v)=\vp_T(v), \quad \vp_T \in \mathcal{C}^{\infty}_c(\R), \, T>0. 
\end{equation} 
To obtain the \emph{forward in time} formulation, we consider $t \mapsto T-t$, with $t \in [0,T]$, with a slight abuse of notation. This gives $\partial_t\vp-a \cdot\partial_{v}\vp-\Ll^*(\vp)=0$ and the \emph{forward in time} adjoint  of \eqref{eq:LinInBolstrong_Cauchy} then reads
   \begin{equation}\label{eq:adjCauchy2}
	\begin{cases}
		\partial_t\vp-a \cdot \partial_v\vp = \Ll^*(\vp), & \\
		\vp(0,v)=\vp_0(v) \; 
	\end{cases}
\end{equation}
with $\vp_0 \in \mathcal{C}_c^\infty(\R)$.

\subsection{Markov Generator and Markov Semigroup} 

In order to prove the existence and uniqueness of a solution to the Cauchy problem \eqref{eq:LinInBolstrong_Cauchy}, we will construct such a solution relying on the theory of semigroups. More precisely, we will consider \eqref{eq:adjCauchy2} solved by the test function $\vp$, which is the adjoint equation of \eqref{eq:LinInBolstrong_Cauchy}, and which has the form:
\begin{align}
\partial_t \vp = T[\vp],
\end{align}
where $T$ is a linear operator, and we will show that the closure of $T$ is a Markov generator. Therefore, as a consequence of the Hille-Yosida theorem, we will deduce that $T$ generates a semigroup of operators $\left(S(t)\right)_{t\geq 0}$ such that the solution of \eqref{eq:LinInBolstrong_Cauchy} exists, is unique, and is of the form
\begin{align}
\vp(t) = S(t)[\vp_0],
\end{align}
with $\vp_0 = \vp(0)$.\\
\noindent
First, we recall the definition of a Markov generator, which is a particular type of linear mappings acting on the space $\mathcal{C}_0(X)$ of continuous functions vanishing at infinity, defined on a metric space $X$ which is locally compact. This definition, as well as a general introduction to the theory of Markov generators and Markov semigroups on $\mathcal{C}_0$ functions defined on $X$ locally compact, can be found \cite{Liggett} (in particular, see Section 3). In the case of the present article, we will consider $X = \mathbb{R}^d$. It is worth to point out that the case of $X$ locally compact requires to develop a more sophisticated theory than in case of $X$ compact.
\begin{defn}
\label{def:Gen} 
    Let $X$ be a locally compact metric space. A linear operator $T: \D \mathscr(T) \subseteq \mathcal{C}_0(X) \rightarrow \mathcal{C}_0(X)$, where $\D \mathscr(T)$ denotes the domain of $T$, is called a Markov generator if 
    \begin{enumerate}
        \item $\mathscr{D}(T)$ is dense in $\mathcal{C}_0(X)$,
        \item there exists $\lambda_0 > 0$ such that for all $\lambda \in\ ]0,\lambda_0]$, there exists a sequence $\left(\vp_n(\lambda)\right)_n \subseteq \D(T)$ that satisfies, denoting by $\psi_n=\psi_n(\lambda)$ the sequence $\psi_n = \vp_n-\lambda T[\vp_n]$:
        \begin{equation}
        \sup_{n\in \mathbb{N}}\|\psi_n\|_{\infty} < +\infty \hspace{5mm} \text{and} \hspace{5mm} \vp_n, \psi_n\ \text{converge to}\ 1\ \text{pointwise as}\ n \rightarrow+ \infty,
        \end{equation}
        \item for any $\vp \in \D(T)$ and $\lambda \geq 0$, denoting $\psi$ by $\psi=\vp-\lambda T[\vp]$, we have
        \begin{equation}
            \inf_{x \in X} \vp(x) \geq \inf_{x \in X} \psi(x),
        \end{equation}
        \item there exists $\overline{\lambda} > 0$ such that the range of $\text{id} - \overline{\lambda}T$ (where $\text{id}$ denotes the identity mapping) is equal to $\mathcal{C}_0(X)$, that is, $\text{id} - \overline{\lambda}T$ is surjective,
    \end{enumerate}
\end{defn}

\begin{remark}
If there exists a positive $\bar \lambda > 0$ as in the fourth point of the previous definition, then  the image of $\text{id} - \lambda T$ is equal to $\mathcal{C}_0(\Omega)$ for every $\lambda>0$.
\end{remark}

\noindent
We now  turn to the definition of a Markov semigroup.

\begin{defn}\label{def:markovSem}
    Let $X$ be a locally compact metric space. A Markov semigroup $\left(S(t)\right)_{t \geq 0}$ is a family of linear operators $S(t):\mathcal{C}_0(X) \rightarrow \mathcal{C}_0(X)$ such that:
    \begin{enumerate}
        \item $S(0)=\text{id}$,
        \item $S(t+s)=S(t)S(s)$ for every $t,s \geq 0$,
        \item $\lim_{t \rightarrow 0^+} \|S(t)[\vp]-\vp\|_{\infty}=0$ for every $\vp \in \mathcal{C}_0(\Omega)$,
        \item $S(t)\vp \geq 0$ whenever $\vp \geq 0$,
        \item there exists a sequence $(\vp_n)_n \subseteq \mathcal{C}_0(X)$ with $\sup_{n \in \mathbb{N}}\|\vp_n\|<+\infty$ such that such $S(t)[\vp_n]\rightarrow 1$ pointwise as $n \rightarrow+ \infty$ for every $t \geq 0$.
    \end{enumerate}
\end{defn}
\smallskip

\noindent
We will rely on the following classical result.
\begin{thm}[Hille-Yosida]\label{thm:HY}
    There is a one-to-one correspondence between Markov generators and Markov semigroups, namely every Markov generator is the generator (in the sense of (3.16) in Theorem 3.16 of \cite{Liggett}) of a Markov semigroup on $\mathcal{C}_0(\R)$.
\end{thm}

\noindent
Now that we have introduced the fundamental tools required for our analysis, we focus on the study of the linear operator associated to the Cauchy problem \eqref{eq:LinInBolstrong_Cauchy}.
\noindent
To this aim, we introduce the linear operator $\left(T,\mathscr{D}(T)\right)$ defined as:
\begin{equation}\label{def:opA}
   T:\D(T) \subseteq \mathcal{C}_0(\R) \rightarrow \mathcal{C}_0(\R), \qquad  T[\vp] :=a \cdot \partial_{v} \vp + \Ll^*[\vp] 
\end{equation}
with domain
\begin{equation}\label{def:DomA}
    \D(T)= \Big\{ \vp \in \mathcal{C}^1(\R) \ /\  \vp, \,
    \partial_v \vp \in \mathcal{C}_0(\R) \Big\}.
\end{equation}

\noindent
We first prove that $T$ maps $\D(T)$ into $\mathcal{C}_0(\mathbb{R}^d)$. 
We observe that the operator $\Ll^*_+$ maps $\mathcal{C}_0(\R)$ into $\mathcal{C}_0(\R)$ into itself. Indeed, since $\vert v' \vert \geq r \vert v \vert$, for any $\vp \in \mathcal{C}_0(\R)$ and $\varepsilon > 0$ there exists $M \geq 0$ such that, for any $v \in \mathbb{R}^d$ with $\vert v \vert \geq M$ and $\omega \in \mathbb{S}^{d-1}$, $
\vert \vp(v') \vert \leq \varepsilon\, .$
This implies, for $\vert v \vert \geq M$, 
\begin{equation}\notag
\left\vert \Ll^*_+[\vp] \right\vert(v) \leq \int_{\mathbb{S}^{d-1}}b(|N \cdot \omega|)\vert \vp(v') \vert \dd \omega  \leq \varepsilon \int_{\mathbb{S}^{d-1}} b(|N \cdot \omega|) \dd \omega,
\end{equation}
and so $\Ll^*_+[\vp] \in \mathcal{C}_0(\mathbb{R}^d)$. Observe that we used the assumption $r \neq 0$ in a crucial manner.
It is also clear that $a \cdot \partial_v \vp \in \mathcal{C}_0(\mathbb{R}^d)$ and $\Ll^*_-[\vp] \in \mathcal{C}_0(\mathbb{R}^d)$ if $\vp \in \D(T)$. As a consequence, $T$ indeed maps $\D(T)$ into $\mathcal{C}(\R)$.\\
\newline
\noindent
In order to prove Theorem \ref{thm:wp1}, we will show that the operator $T$ defined in \eqref{def:opA}-\eqref{def:DomA} generates a Markov semigroup. To do so, we will need the following technical results.

\begin{lem}[Test functions]
\label{LEMMETest_Funct}
For any positive integer $n \geq 1$, the function $\vp_n:\mathbb{R}^d \rightarrow \mathbb{R}$ defined as:
\begin{align}
\label{EQUAT_TestFonctMarkovGene}
\vp_n(v) = \mathds{1}_{\vert v \vert < n} e^{-\frac{n}{n^2-\vert v \vert^2}} \hspace{5mm} \forall v \in \mathbb{R}^d,
\end{align}
satisfies:
\begin{itemize}
\item $\vp_n$ belongs to $\mathcal{C}^1(\mathbb{R}^d)$, and $\vp_n$ and its first derivative $\partial_v \vp_n$ are compactly supported,
\item $\sup_{n\in \mathbb{N}}\|\vp_n\|_{\infty} < +\infty$ and $\sup_{n\in \mathbb{N}}\|\partial_v \vp_n\|_{\infty} < +\infty$,
\item $\vp_n \rightarrow 1$ pointwise, $\partial_v \vp_n \rightarrow 0$ pointwise as $n \rightarrow +\infty$,
\item $\vp_n$ is decreasing, in the sense that for $v,w\in\mathbb{R}^d$ with $\vert v \vert \leq \vert w \vert$, we have:
\begin{align}
0 \leq \vp_n(w) \leq \vp_n(v) \leq 1.
\end{align}
\end{itemize}
\end{lem}

\begin{proof}
The fact that $\vp_n$ is a $\mathcal{C}^1$ function is a classical exercise, and the fact that $\vp_n$ and its derivative $\partial_v \vp_n$ are compactly supported is a consequence of the definition.\\
As for a uniform bound, it is clear that $\sup_{n\in \mathbb{N}}\|\vp_n\|_{\infty} = 1$. Concerning its derivative, since we have:
\begin{align}
\partial_v \vp_n(v) = \mathds{1}_{\vert v \vert < n} \frac{-2nv}{\left( n^2 - \vert v \vert^2 \right)^2} e^{-\frac{n}{n^2 - \vert v \vert^2}},
\end{align}
observing also that $\displaystyle{\sup_{\rho \in \mathbb{R}} \rho^2 e^{-\rho}}$ is a finite, universal constant, we find for any $v \in \mathbb{R}^d$ such that $\vert v \vert < n$:
\begin{align}
\vert \partial_v \vp_n \vert(v) \leq \frac{2n \vert v \vert}{\left( n^2 - \vert v \vert^2 \right)^2} e^{-\frac{n}{n^2 - \vert v \vert^2}} \leq 2 \left(\frac{n}{n^2-\vert v \vert^2}\right)^2 e^{-\frac{n}{n^2-\vert v \vert^2}} \leq 2 \sup_{\rho \in \mathbb{R}} \rho^2 e^{-\rho},
\end{align}
so that we determined an upper bound on $\partial_v \vp_n$, uniform in $v$ and $n$.\\
For $v$ fixed and $n$ large enough, we have $\vert v \vert < n/2$, and for such $v$ and $n$:
\begin{align}
e^{-\frac{n}{n^2-\vert v \vert^2}} \rightarrow 1 \ \text{and}\ \left\vert \frac{-2nv}{\left(n^2 - \vert v \vert^2\right)^2} \right\vert \leq \frac{n^2}{\left(3n^2/4\right)^2} \rightarrow 0 \ \text{as}\ n \rightarrow +\infty.
\end{align}
Therefore, $\vp_n \rightarrow 1$ and $\partial_v \vp_n \rightarrow 0$ pointwise as $n \rightarrow +\infty$.\\
Finally, the monotonocity of $\vp_n$ follows directly from the definition.
\end{proof}

\noindent
Besides, for $\vp \in \D(T)$,  $\psi \in \mathcal{C}_0(\mathbb{R}^d)$ and $\lambda > 0$ given, we will consider the following equation:
$$
\vp - \lambda T[\vp] = \psi,
$$
which, more explicitly,  reads 
$$
\vp + \lambda \Ll_-^*[\vp] - \lambda a \cdot \partial_v \vp - \lambda \Ll_+^*[\vp] = \psi
$$
or, equivalently,  
\begin{align}
\label{EQUATSurjeId-LTDiffeVersi}
\left[ 1 + \lambda C_b \right] \vp(v) - \lambda a \cdot \partial_v \vp(v) - \lambda \Ll_+^*[\vp](v) = \psi(v)
\end{align}
where $C_b = \int_{\mathbb{S}^{d-1}} b( \vert N\cdot\omega \vert) \dd \omega$. Introducing the characteristics $z_v$ satisfying:
\begin{align}
\left\{
\begin{array}{ccc}
\frac{\dd}{\dd t} z_v(t) &=& -\lambda a,\\
z_v(0) &=& v,
\end{array}
\right.
\end{align}
\eqref{EQUATSurjeId-LTDiffeVersi} can be rewritten, for any $\tau \leq s$, as
\begin{align*}
\frac{\dd}{\dd s} \left[ \vp(z_v(s)) e^{\int_\tau^s [1+\lambda C_b ] \dd u} \right] - \lambda \Ll_+^*[\vp](z_v(s)) e^{\int_\tau^s [1+\lambda C_b ] \dd u} = \psi(z_v(s)) e^{\int_\tau^s [1+\lambda C_b ] \dd u}.
\end{align*}
Integrating along the characteristics, that is, integrating the previous equation in $s$ between $t_0 \leq t$, we obtain
\begin{align*}
\vp(v) &= \vp(z_v(t_0-t)) e^{-\int_{t_0}^t [1 + \lambda C_b] \dd u} + \lambda \int_{t_0}^t \Ll_+^*[\vp] (z_v(s-t)) e^{-\int_s^t [1+\lambda C_b ] \dd u} \dd s \nonumber\\
&\hspace{70mm}+ \int_{t_0}^t \psi(z_v(s-t)) e^{-\int_s^t [1+\lambda C_b ] \dd u} \dd s.
\end{align*}
The integration bounds $t_0$ and $t$ being arbitrary, we can choose $t = 0$. We can in addition consider the limit $t_0 \rightarrow -\infty$, which will make vanish the first term of the right hand side, using the fact that the exponential term is smaller than $1$, $z_v(t_0) = v - \lambda t_0 a$ and $\vp \in \mathcal{C}_0(\mathbb{R}^d)$. Therefore, we obtain (at least formally):
\begin{align*}
\vp(v) &= \lambda \int_{-\infty}^0 \Ll_+^*[\vp] (z_v(s)) e^{-\int_s^0 [1+\lambda C_b ] \dd u} \dd s + \int_{-\infty}^0 \psi(z_v(s)) e^{-\int_s^0 [1+\lambda C_b] \dd u} \dd s \nonumber\\
&= \lambda \int_0^{+\infty} \Ll_+^*[\vp] (z_v(-s)) e^{-\int_0^s [1+\lambda C_b] \dd u} \dd s + \int_0^{+\infty} \psi(z_v(-s)) e^{-\int_0^s [1+\lambda C_b] \dd u} \dd s,
\end{align*}
where in the last step we performed the two changes of variables $u \rightarrow -u$ and $s \rightarrow -s$.\\
We introduce then the two operators:
\begin{align}
\label{EQUATDefinOpera__A__}
A[\vp] = \int_0^{+\infty} \Ll_+^*[\vp] (z_v(-s)) e^{-\int_0^s [1+\lambda C_b] \dd u} \dd s
\end{align}
and
\begin{align}
\label{EQUATDefinOpera__B__}
B[\psi] = \int_0^{+\infty} \psi(z_v(-s)) e^{-\int_0^s [1+\lambda C_b] \dd u} \dd s.
\end{align}
We now show that the two operators $A$ and $B$ are well-defined on $\mathcal{C}_0(\mathbb{R}^d)$.

\begin{lem}[Domain and image of the operators $A$ and $B$]
\label{LEMMADomaiImageOpera_A_B_}
Let $\lambda > 0$ be a strictly positive number. The two operators $A$ and $B$, respectively defined in \eqref{EQUATDefinOpera__A__} and \eqref{EQUATDefinOpera__B__}, map $\mathcal{C}_0(\mathbb{R}^d)$ into itself.
\end{lem}

\begin{proof}
We first show that, for any $\vp \in \mathcal{C}_0(\mathbb{R}^d)$ and $v \in \mathbb{R}^d$, $A[\vp](v)$ and $B[\vp](v)$ are well-defined quantities.\\
We have
\begin{align}
\left\vert B[\vp] \right\vert (v) \leq \| \vp \|_\infty \int_0^{+\infty} e^{-s} \dd s = \| \vp \|_\infty,
\end{align}
proving that $B[\vp]$ is a uniformly bounded function. As for $A$, we have
\begin{align}
\left\vert A[\vp] \right\vert (v) &\leq \int_0^{+\infty} \left( \int_{\mathbb{S}^{d-1}} b\left( \frac{v+\lambda s a}{\vert v + \lambda s a \vert}\cdot\omega\right) \vert \vp\left( (v+\lambda s a)' \right) \vert \dd \omega \right) e^{-\int_0^s [1+\lambda C_b ] \dd u} \dd s \nonumber\\
&\leq \frac{C_b \| \vp \|_\infty}{\lambda C_b} \int_0^{+\infty} [ 1+\lambda C_b ] e^{-\int_0^s [ 1 + \lambda C_b ] \dd u} \dd s \nonumber\\
&\leq - \frac{\| \vp \|_\infty}{\lambda} \left[ e^{-\int_0^s [1 + \lambda C_b ] \dd u} \right]_0^{+\infty} = \frac{\| \vp \|_\infty}{\lambda},
\end{align}
using the fact that $e^{-\int_0^s [1 + \lambda C_b ] \dd u} \leq e^{-s}$, which is a vanishing quantity in the limit $s \rightarrow +\infty$.\\
Therefore, $A[\vp]$ and $B[\vp]$ are two continuous and bounded functions when $\vp \in \mathcal{C}_0(\mathbb{R}^d)$.\\
It remains to show that $A[\vp]$ and $B[\vp]$ are vanishing at infinity. In the case of $A$, for any $\varepsilon > 0$, there exists $s_0 = s_0(\lambda,\| \vp \|_\infty,\varepsilon) > 0$ large enough such that
\begin{align}
\int_{s_0}^{+\infty} \left( \int_{\mathbb{S}^{d-1}} b\left( \frac{v+\lambda s a}{\vert v + \lambda s a \vert}\cdot\omega\right) \vert \vp\left( (v+\lambda s a)' \right) \vert \dd \omega \right)& e^{-\int_0^s [1+\lambda C_b 
] \dd u} \dd s \nonumber\\
\leq \frac{\| \vp \|_\infty}{\lambda}& e^{-\int_0^{s_0} [1 + \lambda C_b 
] \dd u} \leq \frac{\| \vp \|_\infty}{\lambda} e^{-s_0} \leq \varepsilon/2.
\end{align}
In addition, since $\vp \in \mathcal{C}_0(\mathbb{R}^d)$, there exists $M \geq 0$ such that for all $v \in \mathbb{R}^d$ with $\vert v \vert \geq M$, we have for all $s \in [0,s_0]$ and $\omega \in \mathbb{S}^{d-1}$:
\begin{align}
\vert \vp((v+\lambda s a)') \vert \leq \frac{\lambda \varepsilon}{2}\cdotp
\end{align}
Therefore
\begin{align}
\int_0^{s_0} \left( \int_{\mathbb{S}^{d-1}} b\left( \frac{v+\lambda s a}{\vert v + \lambda s a \vert}\cdot\omega\right) \vert \vp\left( (v+\lambda s a)' \right) \vert \dd \omega \right)& e^{-\int_0^s [1+\lambda C_b ] \dd u} \dd s \nonumber\\
&\leq \frac{\varepsilon}{2} \int_0^{s_0} [ 1 + \lambda C_b ] e^{-\int_0^s [ 1 + \lambda C_b ] \dd u} \dd s \nonumber\\
&\leq \frac{\varepsilon}{2}\cdotp
\end{align}
We have shown therefore that for $\vert v \vert$ large enough, we have $\vert A[\vp] \vert(v) \leq \varepsilon$, proving that $A[\vp] \in \mathcal{C}_0(\mathbb{R}^d)$. Similar arguments applied to the case of $B[\vp]$ provides the same conclusion for this operator, which concludes the proof of the lemma.
\end{proof}

\noindent
In the end, we can consider a more general form of the equation $\vp - \lambda T[\vp] = \psi$, as summarized in the following result.

\begin{prop}[Generalized equation {$\vp - \lambda T[\vp] = \psi$}]
\label{PROPOGeneralizdEquatSurje}
Let $\lambda > 0$ be a strictly positive number, and let $\psi \in \mathcal{C}_0(\mathbb{R}^d)$ be any continuous function vanishing at infinity. 
If $\vp\in \mathcal{C}_0(\mathbb{R}^d)$ belongs to $\D(T)$ and solves the equation:
\begin{align}
\label{EQUATSurjeId-LTDiffeVers2}
\vp - \lambda T[\vp] = \psi,
\end{align}
then $\vp$ solves the equation:
\begin{align}
\label{EQUATSurjeId-LTIntegVersi}
\vp = \lambda A[\vp] + B[\psi],
\end{align}
where the mappings $A$ and $B$ are defined in \eqref{EQUATDefinOpera__A__} and \eqref{EQUATDefinOpera__B__} respectively. 
Conversely, if $\vp\in \mathcal{C}_0(\mathbb{R}^d)$ solves the equation \eqref{EQUATSurjeId-LTIntegVersi} and if $\vp \in \D(T)$ or $\psi \in \D(T)$, then $\vp$ solves the equation \eqref{EQUATSurjeId-LTDiffeVers2}.
\end{prop}

\noindent
The proof of Proposition \ref{PROPOGeneralizdEquatSurje} is obtained by a direct computation.

\noindent
Finally, as a last intermediate step to establish the results we need for the proof of Theorem \ref{thm:wp1}, we will prove that the operator $A$ is bounded, and determine its norm, as a linear operator acting on $\mathcal{C}_0(\mathbb{R}^d)$.

\begin{prop}[Boundedness of the operator $\lambda A$]
\label{PROPOBoundednesOpera__A__}
Let the operator $A$ be defined as in \eqref{EQUATDefinOpera__A__}. Then we have that $\lambda \|A\| <1$.
\end{prop}

\begin{proof}
Let $\vp \in \mathcal{C}_0(\mathbb{R}^d)$. A similar estimate as in the proof of Lemma \ref{LEMMADomaiImageOpera_A_B_}, slightly more precise, shows that, for any $v\in \mathbb{R}^d$:
\begin{align}
\left\vert A[\vp]\right\vert (v) &\leq C_b \| \vp \|_\infty \underbrace{\int_0^{+\infty}  e^{-\int_0^s [ 1 + \lambda C_b ] \dd u} \dd s}_{= I_\lambda(v)}.
\end{align}
We have:
\begin{align}
I_\lambda(v) = \int_0^{+\infty} e^{-\int_0^s [1+\lambda C_b] \dd u} = \int_0^{+\infty} e^{-(1+\lambda C_b)s} \dd s = \frac{1}{1+\lambda C_b},
\end{align}
so that:
\begin{align}
\left\vert \lambda A[\vp] \right\vert(v) \leq \frac{\lambda C_b}{1 + \lambda C_b} \| \vp \|_\infty,
\end{align}
which proves that $\lambda A$ is bounded, with an operator norm strictly smaller than $1$ for each $\lambda >0$. This improves the estimate obtained along the lines of the proof of Lemma \ref{LEMMADomaiImageOpera_A_B_}.
\end{proof}

\subsection{Proof of Theorem \ref{thm:wp1}}

In order to prove Theorem \ref{thm:wp1}, we first establish that the closure $\overline{T}$ of the operator $T$, defined in \eqref{def:opA}, is a Markov generator. This will allow us to solve the adjoint equation \eqref{eq:adjCauchy2} for the test functions $\varphi$. In turn, choosing test functions $\vp$ solving the adjoint equation \eqref{eq:adjCauchy2} will yield a major simplification in the weak form \eqref{eq:weakf} of the original equation \eqref{eq:LinInBolstrong_Cauchy}.

\begin{proof}[Proof of Theorem \ref{thm:wp1}]
We will check the first three assumptions of Definition \ref{def:Gen} in the case of the unbounded operator T defined in \eqref{def:opA}. Then, we will consider its closure $\overline{T}$, for which these assumptions will also hold, and prove finally that the fourth assumption of Definition \ref{def:Gen} holds also for $\overline{T}$.
\begin{enumerate}
    \item Since the domain $\D(T)$ contains the space $\mathcal{C}_c^\infty(\R)$, it is dense in $\mathcal{C}_0(\R)$. Therefore the first point of Definition \ref{def:Gen} holds for $T$.
    \item The sequence of functions $(\vp_n)_n$ given by the expression \eqref{EQUAT_TestFonctMarkovGene} satisfy for $T$ the second condition of Definition \ref{def:Gen}, as a consequence of the different properties established in Lemma \ref{LEMMETest_Funct} for such functions. Then, the second point of Definition \ref{def:Gen} holds for $T$.
    \item We prove now that $T$ satisfies the third point of Definition \ref{def:Gen}. Let $\lambda \geq 0$ be a non-negative real number, and let us consider the only two different cases that can occur.\\
    Let $\vp \in \D(T)$ and let us first assume that $\inf_{v \in \R} \vp(v)$ is not attained at infinity. It means that there exists a compact set $K \subseteq \R$ and $v_0 \in K$ such that
\begin{equation}
    \inf_{v \in \R} \vp(v) = \inf_{v \in K} \vp(v)= \min_{v \in K} \vp(v)= \min_{v \in \R}\vp(v)=\vp(v_0).
\end{equation}
Then let $\psi=\vp-\lambda T[ \vp ]$, we have that
\begin{align}
    \inf_{v \in \R} \psi(v) \leq \psi(v_0)=\vp(v_0)-\lambda a \cdot \partial_v \vp(v_0)-\lambda \Ll^*(\vp)(v_0) \leq \vp(v_0)= \inf_{v \in \R}\vp(v)
\end{align}
where we have used the fact that $\partial_v \vp(v_0)=0$ and that $\Ll^*(\vp)(v_0) \geq 0$.\\
Let us now assume that $\inf_{v \in \R} \vp(v)$ is attained at infinity. By the definition of the infimum, there exists a sequence of points $(v_n)_n \subseteq \R$ such that $|v_n|\rightarrow + \infty$ as $n \rightarrow + \infty$ and such that
\begin{equation*}
    \vp(v_n) \leq \inf_{v \in \R} \vp(v)+\frac{1}{n}. 
\end{equation*}
Let $\psi=\vp-\lambda T[\vp]$, in this case for any $n\in \mathbb{N}^*$ we have
\begin{align}
    \inf_{v \in \R}\psi(v) \leq \psi(v_n) & = \vp(v_n)-\lambda a \cdot \partial_v \vp(v_n)-\lambda\Ll^*(\vp)(v_n) \notag \\
    &\leq \inf_{v \in \R}\vp(v)+\frac{1}{n}-\lambda a \cdot \partial_v \vp(v_n)-\lambda\Ll^*(\vp)(v_n).
\end{align}
Taking the limit as $n \rightarrow + \infty$ the right-hand side converges to $\inf_{v \in \R}\vp(v)$ since $|v_n| \rightarrow + \infty$ and $\partial_v \vp, \Ll^*(\vp) \in \mathcal{C}_0(\R)$. This gives
\begin{equation*}
    \inf_{v \in \R}\psi(v) \leq \inf_{v \in \R} \vp(v).
\end{equation*}
Therefore $T$ satisfies the third point of Definition \ref{def:Gen}.
\item As for the fourth point of Definition \ref{def:Gen}, the surjectivity of $\text{id} - \lambda \overline{T}$, where $\overline{T}$ is the closure of the operator $T$, is obtained in two steps. First, let us observe that the image of $\text{id} - \lambda T$ is dense in $\mathcal{C}_0(\R)$. Indeed, let us consider any $\psi \in \D(T)$. By Lemma \ref{LEMMADomaiImageOpera_A_B_}, the mapping $A$ introduced in \eqref{EQUATDefinOpera__A__} is a linear mapping from $\mathcal{C}_0(\R)$ to itself and, by Proposition \ref{PROPOBoundednesOpera__A__}, $\text{id} - \lambda A$ is invertible. Hence we can find a unique $\vp \in \mathcal{C}_0(\R)$ that solves the generalized equation \eqref{EQUATSurjeId-LTIntegVersi}. Since in addition $\psi$ was assumed to belong to $\D(T)$, by Proposition \ref{PROPOGeneralizdEquatSurje} the regularity of $\psi$ allows to deduce that $\vp \in \D(T)$ as well.\\
Now, $T$ being a Markov pregenerator, that is, $T$ fulfills the three first conditions of Definition \ref{def:Gen} (see \cite{Ligg008}), we deduce that $T$ is closable (relying for instance on Proposition 2.5 in \cite{Ligg008}), and that the image of $\text{id} - \lambda\overline{T}$ is the whole space $\mathcal{C}_0(\R)$ (since the image of $\overline{T}$ is closed, by Proposition 2.6 in \cite{Ligg008}).\\
The fourth point of Definition \ref{def:Gen} holds as well for $\overline{T}$.
\end{enumerate}
To conclude, let us observe that $T$ being a Markov pregenerator, then so is its closure $\overline{T}$. Moreover $\overline{T}$ fulfills also the fourth condition of Definition \ref{def:Gen} and we conclude that $\overline{T}$ is a Markov generator. Therefore, by the Hille-Yosida theorem \ref{thm:HY}, the operator $\overline{T}$ generates a Markov semigroup $S(t)$.\\
Furthermore if $\vp_0 \in \D(T)$ then $S(t)\vp_0=\vp(t,\cdot) \in \mathcal{C}^1$ can be chosen as a test function in Definition \ref{def:weakSol}. Now for every $f_0 \in \M$  and $\vp \in \mathcal{C}_0(\R)$ we define $f\in C([0,+\infty),\M)$ by means of the \emph{duality formula} \begin{equation}\label{eq:dualityFormula}
        \int_{\R}\vp(v)f(t,\dd v)=\int_{\R}\vp(v)S^*(t)f_0(\dd v)=\int_{\R} S(t)\vp(v) f_0(\dd v). 
    \end{equation}
    Due to the existence of the semigroup $S(t)$ this definition of $f(t)$ is meaningful. Let now  $u(t) \in \mathcal{C}^1([0,+\infty),\mathscr{D}(T))$. Notice that $\overline{T}[u(t)]= T[u(t)]$ and that $S(t)(T[u(t)])=T[S(t)u(t)]$ since $u(t)\in \mathscr{D}(T)$ for all $t \geq 0$.
    Set 
    \begin{equation}
        \Phi(t,v)=\partial_t u(t,v)+T[u(t)](v).
    \end{equation}
    One has that
    \begin{equation}\label{xi}
        S(t)\Phi(t,v)=\partial_t(S(t)u(t,v)).
    \end{equation} 
Integrating \eqref{xi} with respect to $f_0$ over $\R$ and integrating over $[0,t] \times \R$ yields
\begin{equation}
    \int_0^t\dd t\int_{\R}f_0(\dd v)S(s)\Phi(s,v)= \int_0^t\dd s\int_{\R}f_0(\dd v)\partial_s(S(s)u(s,v))=\int_{\R}f_0(\dd v)(S(t)u(t,v)-u(0,v)).
\end{equation}
The above equation and the duality formula \eqref{eq:dualityFormula} give
\begin{align}
   \int_{\R} f(t,\dd v)u(t,v) & - \int_{\R} f_0(\dd v)u(0,v)=\int_0^t \dd s \int_{\R}f_0(\dd v)S(s)\Phi(s,v) = \int_0^t \dd t \int_{\R} f(s,\dd v)\Phi(s) \notag \\
   & = \int_0^t \dd s \int_{\R} f(s,\dd v) \big(\partial_s u(s,v)+T[u(s)](v)\big)
\end{align}
which proves that $f(t,\dd v)$ is a weak solution of equation \eqref{eq:LinInBolstrong_Cauchy}, in the sense of Definition \ref{def:weakSol}.\\
\newline
We turn now to the uniqueness of the weak solution of \eqref{eq:LinInBolstrong_Cauchy}. Let $f_1,f_2$ be two weak solutions to \eqref{eq:LinInBolstrong_Cauchy}, in the sense of Definition \ref{def:weakSol},  with the same initial datum $f_0$. Set $g=f_1-f_2$ and notice that $g$ also is a solution to \eqref{eq:LinInBolstrong_Cauchy}, in the sense of Definition \ref{def:weakSol}, with  initial datum $g_0 = 0$. 
Let $u$ be a non-zero test function such that
\begin{equation}
    \int_{\R}g(\bar{t},\dd v)u(\bar t,v) \neq 0 \; \text{for some $\bar t>0$}.
\end{equation}
Then the function $u(t)=S(\bar t-t)u(\bar t,v)$ is the unique solution of 
\begin{equation}\label{cauchyAux}
\begin{cases}
	-\frac{\dd u}{ \dd t}(t)=a \cdot \partial_{v}  u(t)+\Ll^* (u(t)) & \\
	u(\bar t)=\varphi(\bar t), \; \; 0 \leq t \leq \bar t.  
\end{cases}
\end{equation}
Equation \eqref{eq:weakf} then yields
\begin{gather}
    \int_{\R}g(\bar t,\dd v)u(\bar t,v)= 
    \int_0^{\bar t} \int_{\R}g(t,\dd v)\left[\partial_t  u(t,v)+ a \cdot \partial_v  u(t)+\Ll^*u(t)\right]\dd t.
\end{gather}
The right-hand side is identically zero by \eqref{eq:weakf}, hence $g(t) = 0$ for all $t$. We obtained therefore that $f_1 = f_2$, hence the uniqueness.\\
The proof of Theorem \ref{thm:wp1} is complete.
\end{proof}

\section{Existence, Uniqueness and Stability of steady states}\label{sec:SteadyState}

\subsection{Moments estimates}\label{sec:momentEst}
In order to prove Theorem \ref{thm:steadystate} we need some a priori estimates on the moments of the solution to \eqref{eq:LinInBolstrong_Cauchy}. Recall the definition of the moments $M_0(t), M_1(t), M_2(t)$ given in \eqref{eq:M_0}-\eqref{eq:M_1}-\eqref{eq:M_2}. In the classical literature on the Boltzmann equation these quantities represent the main macroscopic observables related to the solution of the Boltzmann equation.\\
From the weak formulation \eqref{eq:weakf} it is possible to derive a system of evolution equations for $M_0,M_1,M_2$ by choosing as test function $\left\{1,v,|v|^2\right\}$ respectively. It is important to note that the functions $\{1,v,|v|^2\}$ are not test functions in the sense of Definition \ref{def:weakSol} and we should consider suitable cut-offs in addition to obtain admissible test functions. However, assuming enough integrability for the solution $f$, the results obtained by relaxing the cut-offs are the same as the one we would obtain formally choosing $\{1,v,|v|^2\}$ as test functions. We refer, for instance, to Remark 6 in \cite{MNV} for a detailed discussion concerning this technical issue. Therefore, to keep the exposition as simple as possible we will take $\{1,v,|v|^2\}$ as test functions in \eqref{eq:weakf}. Hence, using also the collision rule \eqref{EQUAT_Loi_Colli} leads to
\begin{equation}\label{eq:evolutionMacro}
    \begin{cases}
    \vspace{2mm}
        \displaystyle{\frac{\dd M_0 (t)}{\dd t}=0}, & \\
    \vspace{2mm}    \displaystyle{\frac{\dd M_1(t)}{\dd t} = a M_0(t)-(1+r)\int_{\R}f(t,dv)\int_{\mathbb{S}^{d-1}}b(|N \cdot \omega|)(v\cdot \omega)\omega \dd \omega}, & \\
        \displaystyle{\frac{\dd M_2(t)}{\dd t} =2a \cdot M_1(t)-2(1-r^2) \int_{\R}f(t,dv) \int_{\mathbb{S}^{d-1}}b(|N \cdot \omega|) (v \cdot \omega)^2\dd \omega}.
    \end{cases}
\end{equation}

\noindent
The system of equations \eqref{eq:evolutionMacro} can be integrated exactly. This is the purpose of the next propositions.

\begin{prop}\label{prop:estimatesMacrod=2}
    Let $d=2$, and let $a \in \R$ be a fixed vector. Assume that the restitution coefficient $r$ satisfies $0 <r \leq 1$. Let $f\in C([0,+\infty),\M)$ be the unique weak solution to \eqref{eq:LinInBolstrong_Cauchy} given by Theorem \ref{thm:wp1}. Suppose $\{|v|,|v|^2\}f_0 \in \M$. Then, we have that $M_0(t),M_1(t),M_2(t) \in L^{\infty}([0,+\infty))$.
\end{prop}

\begin{proof}
    We start from \eqref{eq:evolutionMacro} and we consider a rotation $R=R_v$ such that $N=Re_1$ (recalling that $N = v/\vert v \vert$) and introduce $\tilde \omega$ defined as $\omega= R \tilde \omega$. For $d=2$ we have $\tilde \omega = (\tilde \omega_1, \tilde \omega_2) = (\cos \theta, \sin \theta)$ with $\theta \in [0,2\pi]$. Therefore we have
    \begin{align}
        \int_{\mathbb{S}^1}b(|N \cdot \omega|)(v \cdot \omega) \omega \dd \omega & = |v| R \int_{\mathbb{S}^1} b(|\tilde \omega_1|)\tilde \omega_1 \tilde \omega\dd \tilde \omega = |v|R \int_0^{2\pi} b(|\cos \theta|) \cos \theta \left(\begin{array}{cc}
             \cos \theta \\
             \sin \theta
        \end{array}\right) \dd \theta \notag \\
        & = |v|\left(\int_0^{2\pi}b(|\cos \theta|) \cos^2 \theta \dd \theta \right) Re_1 = v \int_0^{2\pi}b(|\cos \theta|) \cos^2 \theta \dd \theta \notag \\
        &=2 v \int_{-1}^1 b(x) \frac{x^2}{\sqrt{1-x^2}}\dd x
    \end{align}
    and
    \begin{align}
        \int_{\mathbb{S}^1} b(|N \cdot \omega|) (v \cdot \omega)^2 \dd \omega & = |v|^2\int_{\mathbb{S}^1} b(|\tilde \omega_1|)(\tilde \omega_1)^2 \dd \tilde \omega = |v|^2 \int_{0}^{2\pi} b(|\cos \theta|) \cos^2 \theta \dd \theta \notag \\
        &= 2|v|^2 \int_{-1}^1 b(|x|) \frac{x^2}{\sqrt{1-x^2}}\dd x.
    \end{align}
     Therefore for we obtain the following equations:
    \begin{align}
        M_0(t) & =M_0(0), \\
        \frac{\dd M_1(t)}{\dd t} &= aM_0(0)-2(1+r)\left( \int_{-1}^1 b(|x|) \frac{x^2}{\sqrt{1-x^2}}\dd x \right)M_1(t), \\
        \frac{\dd M_2(t)}{\dd t} & = 2 a \cdot M_1(t)-2(1-r^2)\left( \int_{-1}^1 b(|x|) \frac{x^2}{\sqrt{1-x^2}}\dd x\right)M_2(t).
    \end{align}
    We define $C_1=2(1+r)\left( \int_{-1}^1 b(|x|) \frac{x^2}{\sqrt{1-x^2}}\dd x \right), C_2= 2(1-r^2)\left( \int_{-1}^1 b(|x|) \frac{x^2}{\sqrt{1-x^2}}\dd x \right)$ and we observe that $C_1-C_2>0$. Integrating we get
    \begin{align}
        M_0(t) & = M_0(0), \label{eq:Mass} \\
    M_1(t) & = e^{-C_1 t} M_1(0)+\frac{aM_0(0)}{C_1}\left(1-e^{-C_1t}\right), \label{eq:bulku} \\
    M_2(t) & = e^{-C_2t}M_2(t)+e^{-C_2t} \int_0^t e^{C_2s} a \cdot M_1(s) \dd s \notag  \\
    & = e^{-C_2t}M_2(0)+ \frac{e^{-C_2t}-e^{-C_1t}}{C_1C_2} a \cdot M_1(0)+\frac{|a|^2M_0(0)}{C_1C_2}(1-e^{-C_2t})+ \frac{|a|^2M_0(0)}{C_1(C_1-C_2)}(e^{-C_2t}-e^{-C_1t}). \label{eq:energyEst}
    \end{align}
    Since by hypothesis we have $M_0(0),M_1(0),M_2(0)$ finite, from \eqref{eq:Mass}-\eqref{eq:bulku}-\eqref{eq:energyEst} 
    we obtain that the moments remain bounded for all times. This concludes the proof.    
\end{proof}

\begin{prop}\label{prop:estimatesMacrod=3}
    Let $d=3$, and let $a \in \R$ be a fixed vector. Assume that the restitution coefficient $r$ satisfies $0 <r \leq 1$. Let $f\in C([0,+\infty),\M)$ be the unique weak solution to \eqref{eq:LinInBolstrong_Cauchy} given by Theorem \ref{thm:wp1}. Suppose $\{|v|,|v^2|\}f_0 \in \M$. Then, we have that $M_0(t),M_1(t),M_2(t) \in L^{\infty}([0,+\infty))$.
\end{prop}

\begin{proof}
    The proof is analogous to the one of Proposition \ref{prop:estimatesMacrod=2}. In fact, consider a rotation $R$ such that $n=Re_1$ and changing variable in the integral as $\omega= R\tilde \omega$ we get 
    \begin{align}
        -(1+r) \int_{\mathbb{S}^{d-1}} b(|N \cdot \omega|)(v \cdot \omega)\omega \dd \omega &= -2\pi(1+r) \left(\int_{-1}^1 b(|x|)x^2 \dd x \right)v, \\
       -(1-r^2) \int_{\mathbb{S}^{d-1}}b(|N \cdot \omega|)(n\cdot \omega)^2 \dd \omega &= -2\pi(1-r^2) \left(\int_{-1}^1 b(|x|) x^2 \dd x\right) |v|^2,
    \end{align}
    where in the integrals we have used three dimensional polar coordinates with the North pole aligned on $e_1$. Hence we obtain
    \begin{align}
        M_0(t) & =M_0(0), \\
        \frac{\dd M_1(t)}{\dd t} &= aM_0(0)-2\pi(1+r)\left(\int_{-1}^1 b(|x|)x^2 \dd x \right)M_1(t), \\
        \frac{\dd M_2(t)}{\dd t} & = a \cdot M_1(t)-2\pi(1-r^2)\left(\int_{-1}^1 b(|x|)x^2 \dd x \right)M_2(t).
    \end{align} 
    Setting $\tilde C_1=2\pi(1+r)\left(\int_{-1}^1 b(|x|)x^2 \dd x \right), \tilde C_2= 2\pi(1-r^2)\left(\int_{-1}^1 b(|x|)x^2 \dd x \right)$, we notice that $\tilde C_1-\tilde C_2>0$. Integrating we obtain:
    \begin{align}
        M_0(t) & = M_0(0), \label{eq:Massd=3} \\
    M_1(t) & = e^{-\tilde C_1 t} M_1(0)+\frac{aM_0(0)}{\tilde C_1}\left(1-e^{-\tilde  C_1t}\right), \label{eq:bulkud=3} \\
    M_2(t) & = e^{-\tilde C_2t}M_2(t)+e^{-\tilde C_2t} \int_0^t e^{\tilde C_2s} a \cdot u(s) \dd s \notag  \\
    & = e^{-\tilde C_2t}E_0+ \frac{e^{-\tilde C_2t}-e^{-\tilde C_1t}}{\tilde C_1\tilde C_2} a \cdot M_1(0)+\frac{|a|^2M_0(0)}{\tilde C_1\tilde C_2}(1-e^{-\tilde C_2t})+ \frac{|a|^2M_0(0)}{\tilde C_1(\tilde C_1-\tilde C_2)}(e^{-\tilde C_2t}-e^{-\tilde C_1t}). \label{eq:energyEstd=3}
    \end{align}
    From the explicit solutions \eqref{eq:Massd=3}, \eqref{eq:bulkud=3}, \eqref{eq:energyEstd=3} the proof follows.   
\end{proof}

\begin{remark}
    We notice that the solutions of \eqref{eq:evolutionMacro} obtained in Proposition \ref{prop:estimatesMacrod=2} and Proposition \ref{prop:estimatesMacrod=3} have the same structure and they only differ because of the constants due to the collision kernel $b$ and the dimension $d$. 
\end{remark}

\subsection{Existence of a steady state}\label{ssec:existence}
To prove existence of a steady state we will use the results of Proposition \ref{prop:estimatesMacrod=2} and Proposition \ref{prop:estimatesMacrod=3} which allow us to use Schauder's fixed point theorem to prove the existence of a stationary solution to \eqref{eq:LinInBolstrong_Cauchy} in the sense of Definition \ref{def:stationarySol}. In fact, without loss of generality we choose $M_0(0)=1$ in \eqref{eq:Mass} and \eqref{eq:Massd=3}, and we define the set:
\begin{equation}\label{eq:U}
    \mathscr{U}=\left \{ f \in \M \ /\ \int_{\R} f(\dd v)=1,  \int_{\R} |v|f(\dd v) \leq a C_0, \int_{\R}|v|^2f(\dd v) \leq 
    |a|^2 \hat C_0 \right\}
\end{equation}
with $C_0, \hat C_0>0$. Notice that this definition of $\mathscr{U}$ covers both cases of $d=2$ and $d=3$. The following is the main theorem of this subsection.

\begin{thm}\label{thm:ssMM}
   Let $d=2,3$, let $a \in \R$ be a fixed vector. Assume that the restitution coefficient $r$ satisfies $0 < r < 1$. Let $f_0 \in \M$ be such that $\{|v|,|v|^2\}f_0 \in \M$ and let $f \in \mathcal{C}([0,+\infty),\M)$ be the unique weak solution to \eqref{eq:LinInBolstrong_Cauchy} obtained in Theorem \ref{thm:wp1}. Then, there exists a stationary solution $f_{\infty}$ of \eqref{eq:weakf} in the sense of Definition \ref{def:stationarySol}, or equivalently, a fixed point for the adjoint $S^*(t)$ of the semigroup $S(t)$ defined by the duality formula \eqref{eq:dualityFormula}.
\end{thm}

\begin{proof}
Due to Proposition \ref{prop:estimatesMacrod=2} and Proposition \ref{prop:estimatesMacrod=3}, the set $\mathscr{U}$ is convex and closed in the weak$-*$ topology of $\M$. In addition, $\mathcal{U}$ is metrizable and hence sequentially compact and $\mathscr{U}$ is also weak$-*$ compact. To see this, we observe that $\mathscr{U}$ is contained in the unit ball of $\mathscr{M}(\R)$ and that the space $\mathcal{C}_0(\R)$ is separable. Since $\mathscr{U}$ is weak$-*$-closed by the Banach-Alaoglu's theorem (e.g., \cite{Brezis}), it follows that $\mathscr{U}$ is compact. We have that for any $h \geq 0$ $S^*(h)\mathscr{U} \subseteq \mathscr{U}$, hence the operator $S^*(h)$ is weak$-*$ compact. Therefore we can apply Schauder's fixed point theorem to prove the existence of $f_*^{(h)}$ such that $S^*(h)f_*^{(h)}= f_*^{(h)}$, and we have by the semigroup property that $S^*(mh)f_*^{(h)}= f_*^{(h)}$ for every $m \in \mathbb{N}$. We consider now a sequence $\{h_k\}_k$ such that $h_k \rightarrow 0$ and its corresponding sequence of fixed points $\big( f_*^{(h_k)} \big)_k$. This sequence is compact in $\mathscr{U}$, since $\mathscr{U}$ itself is compact, and taking a sub-sequence if needed we have $f_*^{(h_k)}\rightarrow f_*$ for some $f_* \in \mathscr{U}$ as $k \rightarrow + \infty$. For any $t >0$, there exists a sequence of integers such that $m_kh_k \rightarrow t$ as $k \rightarrow + \infty$. On the one hand, this yields
    $S^*(m_kh_k)f_*^{(h_k)}=S^*(h_k)f_*^{(h_k)}=f_*^{(h_k)} \rightarrow f_*$, while on the other hand:
    \begin{equation}
        S^*(m_kh_k)f_*^{(h_k)}=(S^*(m_kh_k)f_*^{(h_k)}-S^*(t))f_*^{(h_k)}+S^*(t)f_*^{(h_k)}.
    \end{equation}
    By the weak$-*$ continuity of the semigroup $S^*(t)$, it follows that the right-hand side converges to $S^*(t)f_*$, which gives:
    \begin{equation}
        \lim_{k \rightarrow +\infty} S^*(t)f_*= \lim_{k \rightarrow +\infty} S^*(m_kh_k)f_*^{(h_k)} = \lim_{k \rightarrow +\infty} f_*^{(h_k)}  = f_*
    \end{equation}
    for every $t \geq 0$. Therefore $f_*$ is a fixed point for $S^*(t)$ with $t \geq 0$ arbitrary, and the proof is concluded.  
\end{proof}

\subsection{Uniqueness of the steady state}
\label{ssec:uniqueness}
In this subsection we prove the uniqueness of the stationary solution obtained in Theorem \ref{thm:ssMM}. We will prove uniqueness expressing the stationary solution $f_{\infty}$ by expanding the equation $f_{\infty}=S^*(t)f_{\infty}$ by means of a Neumann series which we then prove to be convergent. This is a suitable adaptation of a well-known strategy, see for instance \cite{BNP}. The main theorem of this subsection is the following.

\begin{thm}\label{thm:uniqStatSol}
   Let $d=2$ or $3$, and let $a \in \R$ be a fixed vector. Assume that the restitution coefficient $r$ satisfies $0<r<1$. Let $f_0 \in \M$ be such that $\{|v|,|v|^2\}f_0 \in \M$. Then, the stationary solution obtained in Theorem \ref{thm:ssMM} is unique for every $t>0$.
\end{thm}

\begin{proof}
    The semigroup $S(t)$ introduced in the proof of Theorem \ref{thm:wp1} acts on $\mathcal{C}_0(\R)$ as
\begin{equation}\label{eq:actionS}
    S(t)f=F(t)\vp+\int_0^tF(t-s)(\Ll^*_+(S(s)\vp))\dd s, \quad \vp \in \mathcal{C}_0(\R)
\end{equation}
where $F(t)$ is the semigroup generated by the operator $E[\vp]=-a \cdot \partial_v\vp-C_b \vp$, where $C_b = \int_{\mathbb{S}^{d-1}} b\big( \vert N \cdot \omega \vert \big) \dd \omega$, with domain $\D(T)$. The explicit action of $F(t)$ on $\mathcal{C}_0(\R)$ is given by
\begin{equation}\label{eq:freestreaming}
    F(t)f=e^{-C_b t}f(v-at).
\end{equation}
    From Theorem \ref{thm:ssMM} there exists $f_{\infty} \in \mathcal{C}([0,+\infty),\M)$ such that 
    \begin{equation}\label{eq:fixedPoint}
        f_{\infty}=S^*(t)f_{\infty}.
    \end{equation}
    From \eqref{eq:actionS} the adjoint semigroup $S^*(t)$ acts as
    \begin{equation}\label{eq:actionS^*}
        S^*(t)f_{\infty}=F^*(t)f_{\infty}+(V(S(t))^*f_{\infty}, \quad f_{\infty} \in \mathscr{M}(\R)
    \end{equation}
    where we have set
    \begin{equation}\label{eq:Volterra}
        V(S(t))\vp=\int_0^t F(t-s)\Ll^*_+(S(s)\vp)\dd s, \quad \vp \in \mathcal{C}_0(\R).
    \end{equation}
    To compute the adjoint of $V(S(t))$ we consider $f \in \mathscr{M}(\R),\vp \in \mathcal{C}_0(\R)$, and we have that
    \begin{align}
        \langle f,V(S(t))\vp\rangle & = \int_{\R} \int_0^t f(\dd v) F(t-s)(\Ll_+^*(S(s)\vp(v))\dd s = \int_0^t \int_{\R} f(\dd v) F(t-s)(\Ll_+^*(S(s)\vp(v)) \dd s \notag \\
        & = \int_0^t \int_{\R} S^*(s)(\Ll_+(F^*(t-s)f(\dd v)))\vp(v)\dd s = \langle (V(S(t))^*f,\vp \rangle        
    \end{align}
    where $\langle \, , \rangle$ denotes the duality bracket. Hence from \eqref{eq:actionS^*} we can write
    \begin{equation}\label{eq:fStatIntermediate}
        f_{\infty}=R(1,F^*(t))\int_0^t S^*(s)(\Ll_+(F^*(t-s)f_{\infty}))\dd s
    \end{equation}
    where $R(1,F(t)^*)$ is the resolvent of $F(t)^*$ at $1$. Notice that from \eqref{eq:freestreaming} we have that 
    $$\|F^*(t)\|=\|F(t)\| \leq e^{-C_b t}<1$$
    for every $t>0$, therefore
    \begin{equation}
        R(1,F^*(t))= \sum_{n=0}^{+\infty} (F^*(t))^n
    \end{equation}
    where the series representation of $R(1,F^*(t))$ is absolutely convergent for any $t>0$. Hence \eqref{eq:fStatIntermediate} reads
    \begin{equation}\label{eq:statSolSeries}
        f_{\infty}=\sum_{n=0}^{+\infty} (F^*(t))^n(V(S(t)))^*f_{\infty}.
    \end{equation}
    The general term of the series can be bounded in the operator norm as
    \begin{align}
        \|(F^*(t))^n(V(S(t)))^*\| & \leq \|F^*(t)\|^n \Big\| \int_0^t S^*(s)(\Ll_+(F^*(t-s)))\dd s \Big \| \notag \\
        &\leq e^{-C_b n t} \sup_{t \geq 0} \|S(t)\| \|\Ll_+\|\int_0^t\|F^*(t-s)\|\dd s \notag \\
        & \leq e^{- C_b n t} C_b \int_0^t e^{-C_b(t-s)}\dd s=e^{-C_b n t}(1-e^{-C_b t}).
    \end{align}
    where we used that fact that $\|S(t)\| \leq 1$  for every $t \geq 0$ since $S(t)$ is a Markov semigroup and that $\|\Ll_+\| \leq C_b$. This gives 
    \begin{align}
       \Bigg \| \sum_{n=0}^{+\infty} (F^*(t))^n(V(S(t)))^* \Bigg \| & \leq (1-e^{-C_b t})\sum_{n=0}^{+\infty}  e^{-n C_b t} = (1-e^{-C_b t}) \frac{1}{1-e^{-C_b t}}
    \end{align}
     for any $t>0$ and shows that the series is converging. Since the Neumann series \eqref{eq:statSolSeries} identifies a unique element we  conclude that the stationary solution $f_{\infty}$ is unique.
\end{proof}

\section{Long time behaviour of the solutions: global attractivity of the steady state}\label{sec:stability}

\subsection{Strategy to prove the long time behaviour}

In this section, we will discuss the long-time behaviour of the measure-valued solutions to \eqref{eq:LinInBolstrong_Cauchy} obtained in Theorem \ref{thm:wp1}.  
To this aim, we will rely on the well developed machinery available for the study of the Boltzmann equations in the case of Maxwell molecules by means of the Fourier transform.  
Initially introduced by A. Bobylev in a nonlinear spatially homogeneous context, this approach has been recently applied in \cite{BoNV020} to  study homoenergetic solutions of the elastic, non linear Boltzmann equation and, specifically, to  prove existence, uniqueness, and stability of a self-similar profile for the Boltzmann equation under the assumption of small shear deformations.

\noindent
For this purpose, we will consider the Fourier transform with respect to the velocity $v$ of such solutions, defined as $\mathcal{F}(t,\xi) = \int_{\R} f(t,v) e^{-i v\cdot \xi} \dd v$.
Since the mass of the solutions is preserved along time, we can consider probability measures for all time, and therefore the Fourier transform $ \mathcal{F}(t,\xi)$ of the solution $f(t,v)$ satisfies $\mathcal{F}(t,0) = 1$ for all times $t$ and
\begin{align}
\left\vert \mathcal{F}(t,\xi) \right\vert = \left\vert \int_{\R}f(t,v) e^{-i v\cdot\xi} \dd v \right\vert \leq \int_{\R} f(t,v) \dd v = 1,
\end{align}
since $f \geq 0$. We recall that in probability theory, the Fourier transforms of probability measures are often referred to as \emph{characteristic functions}, see e.g. \cite{Fe2}. The characteristic functions form a
subset $\Phi$ of the space of the complex-valued continuous functions $\mathcal{C}(\R,\mathbb{C})$. First we consider the evolution equation of the Fourier transform $\mathcal{F}$, which reads
\begin{align}
\label{EQUATFouriFirstVersi}
\partial_t \mathcal{F}(t,\xi) + \left(1+i(a\cdot\xi)\right) \mathcal{F}(t,\xi) = \int_{ \mathbb{S}^{d-1}} \widetilde{b} \left( \left\vert \frac{\xi}{\vert \xi \vert}\cdot\sigma  \right\vert\right) \mathcal{F}(t,\overline{\xi})\dd \sigma,
\end{align}
where $$\displaystyle
\overline{\xi} = \frac{1-r}{2} \xi + \frac{1+r}{2} \vert \xi \vert \sigma\, 
$$
We refer to \cite{CCC} for more details on the expression of $\overline{\xi}$, and to Appendix  \ref{app:evolFourier} for the detailed justification of the evolution equation for the Fourier transform \eqref{EQUATFouriFirstVersi}. 

\noindent
We will then proceed according to the following steps.
\begin{itemize}
\item We will prove that initial data $f_0$ of the original evolution equation \eqref{eq:LinInBolstrong_Cauchy} can be properly compared to the steady state in Fourier variables. Namely, we will show that if $f_0$ has finite first order moments, as it is the case for the steady state $f_\infty$, then, their respective Fourier transforms $\mathcal{F}_0$ and $\mathcal{F}_\infty$ can be compared as follows 
\begin{align}
\left\vert \mathcal{F}_0(\xi) - \mathcal{F}_\infty(v) \right\vert \leq C \vert \xi \vert, \quad \text{for all $\xi \in \R$ with $C>0$.}
\end{align}
\item Considering an initial datum $\mathcal{F}_0$ for \eqref{EQUATFouriFirstVersi} satisfying
\begin{align}
\sup_{\xi \in \R} \frac{\left\vert \mathcal{F}_0(\xi) \right\vert}{1 + \vert \xi \vert} <+\infty,
\end{align}
we will prove that \eqref{EQUATFouriFirstVersi} is globally well-posed in $\mathcal{C}([0,+\infty),\mathcal{C}_1)$ where
\begin{align}
\mathcal{C}_1 = \{ \mathcal{F} \in \mathcal{C}_0(\mathbb{R}^d,\mathbb{C})  \, : \, \ \vert \hspace{-0.5mm}\vert \mathcal{F} \vert \hspace{-0.5mm}\vert_{\infty,1} < +\infty\},
\end{align}
with
\begin{align}\label{eq:normC_1}
\vert \hspace{-0.5mm}\vert \mathcal{F} \vert \hspace{-0.5mm}\vert_{\infty,1} = \sup_{\xi \in \mathbb{R}^d} \frac{\vert \mathcal{F}(\xi) \vert}{1 + \vert \xi \vert} \cdotp
\end{align}
\item In addition, we will prove a comparison principle for the evolution equation \eqref{EQUATFouriFirstVersi}. More precisely, after defining an appropriate notion of super-solution for \eqref{EQUATFouriFirstVersi}, we will prove that if $\mathcal{F}$ is a solution of \eqref{EQUATFouriFirstVersi} with initial datum $\mathcal{F}_0$ and if $\mathcal{G}$ is a super-solution with initial datum $\mathcal{G}_0$ with $\vert \mathcal{F}_0(\xi) \vert \leq \mathcal{G}_0(\xi)$ for all $\xi$, then we have $\vert \mathcal{F}(t,\xi) \vert \leq \vert \mathcal{G}(t,\xi) \vert$, for all $t \geq 0$ and $\xi \in \R$.
\item We will finally conclude by proving that there exist super-solutions $\mathcal{G}$ such that $\vert \mathcal{G}(t,\xi) \vert = \varphi_p(t) \vert \xi \vert^p$, where $\varphi_p(t)$ is a function vanishing at infinity with $p>0$. This will allow us to deduce that
\begin{align}
\vert \mathcal{F}(t,\xi) - \mathcal{F}_\infty(v) \vert \leq \varphi_1(t) \vert \xi \vert,
\end{align}
hence the local uniform convergence of $\mathcal{F}$ towards the steady state $\mathcal{F}_\infty$ as $t \rightarrow +\infty$, and so, the weak convergence of the initial solution $f$ of the equation \eqref{eq:LinInBolstrong_Cauchy} towards the steady state $f_\infty$ we uniquely determined in Section \ref{sec:SteadyState}.
\end{itemize}

\subsection{Comparison of the Fourier transforms of two functions with same mass and finite first moments}

We first provide a general estimate concerning functions with finite zero-th and first moments. This estimate will be useful to compare the initial datum for \eqref{EQUATFouriFirstVersi} of the Fourier transform of a solution $f$ of \eqref{eq:LinInBolstrong_Cauchy}, with the Fourier transform $\mathcal{F}_\infty$ of the steady state $f_\infty$ determined in Section \ref{sec:SteadyState}. This result is a direct adaptation of a more general argument that the reader may find for instance in \cite{BoNV020} (see Lemma 5.1 page 421).\\
\noindent
We will rely on the following result.

\begin{lem}

\label{LEMMAUnifoBoundExponTail_}
There exists a positive constant $C$ such that, for any $v,\xi \in \mathbb{R}^d$, we have 
\begin{align}
\left\vert \sum_{n = 2}^{+\infty} \frac{\left(-iv\cdot\xi\right)^n}{n!} \right\vert \leq C \min\left( \vert v \vert  \vert \xi \vert, \vert v \vert^2  \vert \xi \vert^2 \right).
\end{align}
\end{lem}

\noindent
The proof, which is elementary, is obtained by distinguishing the cases $|v||\xi|<1$ and $|v||\xi| \geq 1$.

\begin{prop}[Comparison principle]\label{PROPOCompaFouriFinitFirstMomnt}
Let $f,g \in \mathscr{M}(\R)$ such that
\begin{align}
\int_{\R} f(\dd v) = \int_{\R} g(\dd v) = 1,
\end{align}
\begin{align}
M_{1,f} = \int_{\R} \vert v \vert f(\dd v) < +\infty, \hspace{5mm} M_{1,g} = \int_{\R} \vert v \vert g(\dd v) < +\infty.
\end{align}
Then, there exists a positive constant $C$, depending on $M_{1,f},M_{1,g}$, such that for any $\xi \in \mathbb{R}^d$ we have
\begin{align}
\left\vert \mathcal{F}(f)(\xi) - \mathcal{F}(g)(\xi) \right\vert \leq C \vert \xi \vert.
\end{align}
where $\mathcal{F}_f$ and $\mathcal{F}_g$ denote respectively the Fourier transforms of the measure $f$ and $g$.
\end{prop}

\begin{proof}
By definition, we have:
\begin{align}
\left\vert \mathcal{F}(f)(\xi) - \mathcal{F}(g)(\xi) \right\vert & = \left\vert \int_{\R} e^{-i v\cdot\xi} f(\dd v) - \int_{\R} e^{-i v\cdot \xi}  g(\dd v) \right\vert \nonumber\\
&\leq \underbrace{\left\vert \int_{\R} f(\dd v) - \int_{\R} g(\dd v) \right\vert}_{=0} + \vert \xi \vert  \left[ \int_{\R} \vert v \vert f(\dd v) + \int_{\R} \vert v \vert g(\dd v) \right] \nonumber\\
&\hspace{45mm}+ \left\vert \int_{\R} \left( \sum_{n = 2}^{+\infty} \frac{\left(- i v \cdot \xi \right)^n}{n !}\right) \left( f(\dd v) - g(\dd v) \right) \right\vert \nonumber\\
&\leq \vert \xi \vert \left[ M_{1,f} + M_{1,g} + C \left( M_{1,f} + M_{1,g} \right) \right]
\end{align}
where $C$ is the constant provided by Lemma \ref{LEMMAUnifoBoundExponTail_}. The proof is then complete.
\end{proof}

\subsection{Well-posedness of the evolution equation of the Fourier transform}

This section is devoted to the study of the global well-posedness for \eqref{EQUATFouriFirstVersi} in the space $\mathcal{C}_1$. To this end we introduce an integrated in time version of \eqref{EQUATFouriFirstVersi}. Let $\mathcal{F}$ be a regular enough solution of \eqref{EQUATFouriFirstVersi}, multiplying the equation by $e^{(1+i(a\cdot\xi))t}$ and integrating with respect to the time variable, we obtain
\begin{align}
\int_0^t \partial_s \left[ \mathcal{F}(s,\xi) e^{(1+i(a\cdot\xi))s} \right] \dd s = \int_0^t \int_{\mathbb{S}^{d-1}} b \left( \left\vert \frac{\xi}{\vert \xi \vert}\cdot \sigma \right\vert \right) \mathcal{F}(s,\overline{\xi}) e^{(1+i(a\cdot\xi))s} \dd \sigma \dd s, 
\end{align}
from which we find
\begin{align}
\label{EQUATFouriIntegTime_Vers1}
\mathcal{F}(t,\xi) = \mathcal{F}(0,\xi) e^{-\left(1+i(a\cdot\xi)\right) t} + \int_0^t \int_{\mathbb{S}^{d-1}} b \left( \left\vert \frac{\xi}{\vert \xi \vert}\cdot \sigma \right\vert \right) \mathcal{F}(s,\overline{\xi}) e^{-(1+i(a\cdot\xi))(t-s)} \dd \sigma \dd s.
\end{align}

\noindent
We introduce the following definition.

\begin{defn}\label{def:GWPfourier}

Let $\mathcal{F}_0 \in \mathcal{C}_1$, we say that $\mathcal{F} \in \mathcal{C}([0,+\infty), \mathcal{C}_1)$ is a \emph{global solution} of \eqref{EQUATFouriIntegTime_Vers1} if for any $t \geq 0$ and $\xi \in \mathbb{C}^d$ we have
\begin{align}
\mathcal{F}(t,\xi) = \mathcal{F}_0(\xi) e^{-\left(1+i(a\cdot\xi)\right) t} + \int_0^t \int_{\mathbb{S}^{d-1}} b \left( \left\vert \frac{\xi}{\vert \xi \vert}\cdot \sigma \right\vert \right) \mathcal{F}(s,\overline{\xi}) e^{-(1+i(a\cdot\xi))(t-s)} \dd \sigma \dd s,
\end{align}
where the space $\mathcal{C}([0,+\infty), \mathcal{C}_1)$ is equipped with the norm
\begin{align}
\vert \hspace{-0.5mm} \vert \hspace{-0.5mm} \vert \mathcal{F} \vert \hspace{-0.5mm} \vert \hspace{-0.5mm} \vert_{\infty,1} = \sup_{t \geq 0 }\vert \hspace{-0.5mm} \vert \mathcal{F}(t,\cdot) \vert \hspace{-0.5mm} \vert_{\infty,1}
\end{align}
and where the norm $\|\cdot\|_{1,\infty}$ has been defined in \eqref{eq:normC_1}.
\end{defn}

\begin{prop}
\label{PROPOGlobaWell-PosedEquatFouriTrnsf}
Let $\mathcal{F}_0 \in \mathcal{C}_1$, then there exists a unique global solution to \eqref{EQUATFouriFirstVersi} in the sense of Definition \ref{def:GWPfourier}.
\end{prop}

\begin{remark}
Observe that the result of Proposition \ref{PROPOGlobaWell-PosedEquatFouriTrnsf} provides the well-posedness result for the inelastic linear Boltzmann equation \eqref{eq:LinInBolstrong_Cauchy}, via a different approach from the semigroup strategy of Section \ref{sec:WP}. The drawbacks of this approach are, on the ond hand more stringent hypotheses, and on the other hand, the fact that one needs to prove that the solution that is obtained for \eqref{EQUATFouriFirstVersi} corresponds indeed to the Fourier transform of a probability measure, as it is carefully done in \cite{BoNV020}.
\end{remark}
\noindent
The proof of Proposition \ref{PROPOGlobaWell-PosedEquatFouriTrnsf} follows the argument presented in \cite{BoNV020}. The solution is constructed via an iterative scheme, which will also allow us to compare the solutions. The drawback of this approach is that the uniqueness has to be proved independently.

\begin{proof}[Proof of Proposition \ref{PROPOGlobaWell-PosedEquatFouriTrnsf}]
Consider the following sequence of functions $\left( \mathcal{F}^{(n)} \right)_{n \geq 0}$, defined as
\begin{align}
\label{EQUATDefinSequence__F^(n)}
\left\{
\begin{array}{rcl}
\mathcal{F}^{(0)}(t,\xi) &=& 0,\\
\mathcal{F}^{(n+1)}(t,\xi) &=& \displaystyle{\mathcal{F}_0(\xi) e^{-\left(1+i(a\cdot\xi)\right) t} + \int_0^t \int_{\mathbb{S}^{d-1}} b \left( \left\vert \frac{\xi}{\vert \xi \vert}\cdot \sigma \right\vert \right) \mathcal{F^{(n)}}(s,\overline{\xi}) e^{-(1+i(a\cdot\xi))(t-s)} \dd \sigma \dd s}.
\end{array}
\right.
\end{align}
Let $T > 0$ to be adjusted later. We introduce the Banach space 
\begin{equation}
    X=\mathcal{C}([0,T],\mathcal{C}_1)
\end{equation}
equipped with norm
\begin{equation}
    \|\mathcal{F}\|_X= \sup_{t \in [0,T]}\|\mathcal{F}(t,\cdotp)\|_{1,\infty}
\end{equation}
By induction we show that $(\mathcal{F}^{(n)})_{n \geq 0} \in X$. In fact, we observe that trivially $\mathcal{F}^{(0)} \in X$ and that, since $\mathcal{F}_0 \in \mathcal{C}_1$, we have $ \vert \mathcal{F}^{(1)}(t,\xi) \vert \leq \vert \mathcal{F}_0(\xi) \vert$, and thus $\mathcal{F}^{(1)} \in X$. Assume now that $\mathcal{F}^{(n)} \in X$, we have that 
\begin{align}
\left\vert \frac{\mathcal{F}^{(n+1)}(t,\xi)}{1 + \vert \xi \vert} \right\vert &\leq \left\vert \frac{\mathcal{F}_0(\xi)e^{-(1+i(a\cdot\xi))t}}{1 + \vert \xi \vert} \right\vert + \left\vert \int_0^t \int_{\mathbb{S}^{d-1}} b \left( \left\vert \frac{\xi}{\vert \xi \vert}\cdot \sigma \right\vert \right) \frac{\mathcal{F^{(n)}}(s,\overline{\xi})}{1+\vert \xi \vert} e^{-(1+i(a\cdot\xi))(t-s)} \dd \sigma \dd s \right\vert \nonumber\\
&\leq \vert \hspace{-0.5mm} \vert \mathcal{F}_0 \vert \hspace{-0.5mm} \vert_{\infty,1} + \|\mathcal{F^{(n)}}\|_X\int_0^t \int_{\mathbb{S}^{d-1}} b \left( \left\vert \frac{\xi}{\vert \xi \vert}\cdot \sigma \right\vert \right) \frac{1 + \vert \overline{\xi} \vert}{1 + \vert \xi \vert} e^{-(t-s)}\dd \sigma \dd s.
\end{align}
We observe that
\begin{align}
\vert \overline{\xi} \vert \leq \frac{1-r}{2} \vert \xi \vert + \frac{1+r}{2} \vert \xi \vert = \vert \xi \vert
\end{align}
from which it follows that
\begin{align}
\label{EQUATCntrlNormeInfty1_0,TVers1}
\left\vert \frac{\mathcal{F}^{(n+1)}(t,\xi)}{1 + \vert \xi \vert} \right\vert &\leq \vert \hspace{-0.5mm} \vert \mathcal{F}_0 \vert \hspace{-0.5mm} \vert_{\infty,1} + \|\mathcal{F}^{(n)}\|_X (1-e^{-t}) \int_{\mathbb{S}^{d-1}} b\left(\frac{\xi}{|\xi|} \cdot \sigma\right) \dd \sigma \leq \|\mathcal{F}_0\|_{\infty,1}+\|\mathcal{F}^{(n)}\|_X.
\end{align}
Thus, taking  the supremum first on $\xi \in \R$ and then on $t\in[0,T]$ in the above equation implies that $\mathcal{F}^{(n+1)} \in X$. By induction, we deduce that $(\mathcal{F}^{(n)})_{n \geq0} \in X$ for every $n \in \mathbb{N}$.\\
\noindent
We show now that the sequence $\left( \mathcal{F}^{(n)} \right)_{n\geq 0}$ is converging in $X$, for $T$ appropriately chosen. We first estimate the distance of two consecutive terms of  $(\mathcal{F}^{(n)})_{n \geq 0}$. In fact, we have that
\begin{align}
\left\vert \frac{\mathcal{F}^{(n+2)}(t,\xi) - \mathcal{F}^{(n+1)}(t,\xi)}{1 + \vert \xi \vert} \right\vert &= \left\vert \int_0^t\int_{\sigma\in\mathbb{S}^{d-1}} b \left( \left\vert \frac{\xi}{\vert \xi \vert}\cdot \sigma \right\vert \right) \frac{\left[ \mathcal{F}^{(n+1)}(s,\overline{\xi}) - \mathcal{F}^{(n)}(s,\overline{\xi}) \right]}{1 + \vert \xi \vert} e^{-(1+i(a\cdot\xi))(t-s)} \dd \sigma \dd s \right\vert \nonumber\\
&\leq \|\mathcal{F}^{(n+1)}-\mathcal{F}^{(n)}\|_X \int_0^t \int_{\sigma\in\mathbb{S}^{d-1}} b \left( \left\vert \frac{\xi}{\vert \xi \vert}\cdot \sigma \right\vert \right) \frac{1 + \vert \overline{\xi} \vert}{1 + \vert \xi \vert} e^{-(t-s)} \dd \sigma \dd s \nonumber\\
&\leq (1-e^{-t})\vert \hspace{-0.5mm} \|\mathcal{F}^{(n+1)}-\mathcal{F}^{(n)}\|_X \leq (1-e^{-T})\|\mathcal{F}^{(n+1)}-\mathcal{F}^{(n)}\|_X,
\end{align}
which yields
\begin{align}
    \|\mathcal{F}^{(n+1)}-\mathcal{F}^{(n)}\|_X \leq \frac{1}{2^n} \|\mathcal{F}^{(1)}-\mathcal{F}^{(0)}\|_X = \frac{1}{2^n} \|\mathcal{F}^{(1)}\|_X
\end{align}
and so, for a given $n_0 \in \mathbb{N}$ and any $p,q \geq n_0$ such that $p \leq q$, we have

\begin{align}
    \|\mathcal{F}^{(q)}-\mathcal{F}^{(p)}\|_X & \leq \sum_{n=p}^{q-1} \|\mathcal{F}^{(n+1)}-\mathcal{F}^{(n)}\|_X \leq \sum_{n=p}^{q-1} \frac{1}{2^n} \|\mathcal{F}^{(1)}\|_X \notag\\
    & \leq \frac{1}{2^{n_0}} \|\mathcal{F}^{(1)}\|_X \sum_{n=0}^{q-p-1} \frac{1}{2^n}= \frac{1}{2^{n_0-1}} \|\mathcal{F}^{(1)}\|_X
\end{align}
which proves that $\left( \mathcal{F}^{(n)} \right)_{n\geq 0}$ is a Cauchy sequence in $X$ and thus it is converging to a certain $\mathcal{F} \in X$. Let us observe that the chosen $T$ does not depend on $\mathcal{F}_0$ and  therefore, we can bootstrap the previous argument, and construct recursively a solution on $[nT,(n+1)T]$, for all $n \in \mathbb{N}$.  Moreover we notice that for $T=\log 2$, \eqref{EQUATCntrlNormeInfty1_0,TVers1} implies that
\begin{equation}
    \|\mathcal{F}^{(n+1)}\|_X \leq \|\mathcal{F}_0\|_{\infty,1}+\frac{1}{2} \|\mathcal{F}^{(n)}\|_X
\end{equation}
from which we deduce by an immediate recursion

\begin{equation}
    \|\mathcal{F}^{(n)}\|_X \leq \sum_{k=0}^n \frac{1}{2^k} \|\mathcal{F}_0\|_{\infty,1}  = 2 \|\mathcal{F}_0\|_{\infty,1}.
\end{equation} 
This last estimate proves that the function $\mathcal{F}$ belongs to the space $X$ for any $T > 0$.\\
\noindent
Now, passing to the limit $n \mapsto +\infty$ in \eqref{EQUATDefinSequence__F^(n)}, we find that $\mathcal{F}$
is a solution in the sense of Definition \ref{def:GWPfourier}. It remains to prove the uniqueness of $\mathcal{F}$. To this end, let us assume that there exists another solution $\mathcal{G} \in X$, with the same initial datum $\mathcal{F}_0$. We begin by claiming that $\|\mathcal{F}\|_X \leq \|\mathcal{F}_0\|_{\infty,1}$. In fact, refining again \eqref{EQUATCntrlNormeInfty1_0,TVers1} we get
\begin{align}\label{EQUATCntrlNormeInfty1_0,TVers2}
    \|\mathcal{F}(t,\cdot)\|_{\infty,1} \leq \|\mathcal{F}_0\|_{\infty,1}e^{-t}+\int_0^t \|\mathcal{F}(s,\cdot)\|_{\infty,1}e^{-(t-s)}\dd s.
\end{align}
This is a Gr\"onwall-type inequality, which can be rewritten as
\begin{align}
\frac{\dd}{\dd t} \left[ \left( \int_0^t \vert \hspace{-0.5mm} \vert \mathcal{F}(s,\cdotp) \vert \hspace{-0.5mm} \vert_{\infty,1} e^s \dd s \right) e^{-t}\right] = \vert \hspace{-0.5mm} \vert \mathcal{F}(t,\cdotp) \vert \hspace{-0.5mm} \vert_{\infty,1} - \left( \int_0^t \vert \hspace{-0.5mm} \vert \mathcal{F}(s,\cdotp) \vert \hspace{-0.5mm} \vert_{\infty,1} e^s \dd s \right) e^{-t} \leq \vert \hspace{-0.5mm} \vert \mathcal{F}_0 \vert \hspace{-0.5mm} \vert_{\infty,1} e^{-t},
\end{align}
from which integrating in time provides
\begin{align}
\label{EQUATEstimIntegNorme__F__Gronw}
\left( \int_0^t \vert \hspace{-0.5mm} \vert \mathcal{F}(s,\cdotp) \vert \hspace{-0.5mm} \vert_{\infty,1} e^s \dd s \right) e^{-t} \leq \vert \hspace{-0.5mm} \vert \mathcal{F}_0 \vert \hspace{-0.5mm} \vert_{\infty,1} \int_0^t e^{-s} \dd s = \vert \hspace{-0.5mm} \vert \mathcal{F}_0 \vert \hspace{-0.5mm} \vert_{\infty,1} \left(1-e^{-t}\right).
\end{align}
Estimating the integral term in \eqref{EQUATCntrlNormeInfty1_0,TVers2} by \eqref{EQUATEstimIntegNorme__F__Gronw}, we find our claim
\begin{align}
\label{EQUATNormeDiffeFouri}
\vert \hspace{-0.5mm} \vert \mathcal{F}(t,\cdotp) \vert \hspace{-0.5mm} \vert_{\infty,1} \leq \vert \hspace{-0.5mm} \vert \mathcal{F}_0 \vert \hspace{-0.5mm} \vert_{\infty,1}.
\end{align}
We turn now to the question of the uniqueness. Let $\mathcal{F}, \mathcal{G} \in X$ be two solutions of \eqref{EQUATFouriIntegTime_Vers1} with the same initial datum $\mathcal{F}_0$. The equation \eqref{EQUATFouriIntegTime_Vers1} being linear, $\mathcal{F}-\mathcal{G}$ is also a solution to \eqref{EQUATFouriIntegTime_Vers1}, with identically zero initial datum. Applying \eqref{EQUATNormeDiffeFouri} to $\mathcal{F}-\mathcal{G}$, we deduce that the difference $\mathcal{F}-\mathcal{G}$ is zero for any time. This proves the uniqueness of the solution in $X$ for arbitrary $T>0$, concluding the proof.
\end{proof}

\subsection{Comparison principle for \eqref{EQUATFouriFirstVersi}}

Since we are considering Fourier transforms of the solutions $f$ of \eqref{eq:LinInBolstrong_Cauchy}, we are working with complex-valued functions in general. In most cases (as it is the case in \cite{BoNV020}), the solutions of evolutionary PDEs remain real-valued if the initial datum is real-valued. This leads to natural definitions of super-solutions, and which enables later to compare the solutions using the monotonocity of appropriate integral operators.\\
In the present case, there is an additional difficulty with respect to the situation studied in \cite{BoNV020}: the drift term in the evolution equation \eqref{eq:LinInBolstrong_Cauchy} forces $f$ to be complex-valued, in the sense that, even if we assume that the initial datum $\mathcal{F}_0$ is real-valued, the associated solution of \eqref{EQUATFouriIntegTime_Vers1} will not be real-valued in general. Moreover, it seems that there is no clear way to consider real-valued solutions of \eqref{EQUATFouriIntegTime_Vers1}, neither to consider transformations of such solutions that would remain real-valued (such as multiplying the solutions by $e^{i(a\cdot\xi)t}$).\\
As a consequence, we will be led to introduce a weaker notion of super-solution. This notion is weaker in the sense that a solution of \eqref{EQUATFouriIntegTime_Vers1} will not be a super-solution under the forthcoming definition, although our notion involves a more demanding inequality than the classical one. Moreover, since real-valued solutions cannot be considered, no comparison principle between two ordered solutions seems to hold.
Nevertheless, we will establish a comparison principle between a solution and a super-solution, provided that the former initially lies below the latter.

\begin{defn}[Super-solutions for \eqref{EQUATFouriIntegTime_Vers1}]
\label{DEFINSuper-SoluEquatFouri}
Let $\mathcal{G}_0 \in \mathcal{C}(\mathbb{R}^d,\mathbb{R}_+)$. We say that $\mathcal{G} \in \mathcal{C}\left([0,+\infty),\mathcal{C}(\mathbb{R}^d,\mathbb{R}_+)\right)$ is a \emph{super-solution to the evolution equation \eqref{EQUATFouriIntegTime_Vers1}  with initial datum $\mathcal{G}_0$} if, for all $t \geq 0$ and $\xi \in \mathbb{R}^d$, we have
\begin{align}
\label{EQUATDefinSuper-SoluEquatFouri}
\mathcal{G}(t,\xi) \geq \mathcal{G}_0(\xi) e^{-t} + \int_0^t \int_{\mathbb{S}^{d-1}} b \left( \left\vert \frac{\xi}{\vert \xi \vert}\cdot\sigma \right\vert \right) \mathcal{G}(s,\overline{\xi}) e^{-(t-s)} \dd \sigma \dd s.
\end{align}
\end{defn}

\noindent
We now prove the comparison principle between a solution of \eqref{EQUATFouriIntegTime_Vers1}, and a super-solution in the sense of Definiton \ref{DEFINSuper-SoluEquatFouri}, assuming that the initial data are ordered.

\begin{prop}[Comparison principle]\label{PROPOCompaPrincSuperSolutFouri}
Let $\mathcal{F}_0 \in \mathcal{C}_1$ and $\mathcal{G}_0 \in \mathcal{C}(\mathbb{R}^d,\mathbb{R}_+)$ such that for all $\xi \in \mathbb{R}^d$:
\begin{align}
\left\vert \mathcal{F}_0(\xi) \right\vert \leq \mathcal{G}_0(\xi).
\end{align}
Consider the solution $\mathcal{F}$ to \eqref{EQUATFouriFirstVersi} obtained in Proposition \ref{PROPOGlobaWell-PosedEquatFouriTrnsf} with initial datum $\mathcal{F}_0$ and assume that there exists a super-solution to \eqref{EQUATFouriIntegTime_Vers1} in the sense of Definition \ref{DEFINSuper-SoluEquatFouri}, with initial datum $\mathcal{G}_0$. Then, for all $t \geq 0$ and $\xi \in \mathbb{R}^d$, we have
\begin{align}
\left\vert \mathcal{F}(t,\xi) \right\vert \leq \mathcal{G}(t,\xi).
\end{align}
\end{prop}

\begin{proof}
We consider again the sequence of functions introduced in \eqref{EQUATDefinSequence__F^(n)}, defined recursively as
\begin{align}
\left\{
\begin{array}{rcl}
\mathcal{F}^{(0)}(t,\xi) &=& 0,\\
\mathcal{F}^{(n+1)}(t,\xi) &=& \displaystyle{\mathcal{F}_0(\xi) e^{-\left(1+i(a\cdot\xi)\right) t} + \int_0^t \int_{\mathbb{S}^{d-1}} b \left( \left\vert \frac{\xi}{\vert \xi \vert}\cdot \sigma \right\vert \right) \mathcal{F^{(n)}}(s,\overline{\xi}) e^{-(1+i(a\cdot\xi))(t-s)} \dd \sigma \dd s}.
\end{array}
\right.
\end{align}
According to the proof of Proposition \ref{PROPOGlobaWell-PosedEquatFouriTrnsf}, we know that $\left(\mathcal{F}^{(n)}\right)_{n\geq 0}$ converges towards the unique solution $\mathcal{F}$ of \eqref{EQUATFouriIntegTime_Vers1} with initial datum $\mathcal{F}_0$. Since $\mathcal{G}$ is assumed to be non-negative, we have by assumption $\left\vert \mathcal{F}^{(0)}(t,\xi) \right\vert \leq \mathcal{G}(t,\xi)$ for all $t \geq 0$ and $\xi \in \mathbb{R}^d$. Now assume that for a certain index $n \in \mathbb{N}$ we have that
\begin{align}
\label{EQUATComparisonPropa}
\left\vert \mathcal{F}^{(n)}(t,\xi) \right\vert \leq \mathcal{G}(t,\xi),
\end{align}
for all $t \geq 0$ and $\xi \in \mathbb{R}^d$. Then, using the definition of $\mathcal{F}^{(n+1)}$, we deduce that
\begin{align}
\left\vert \mathcal{F}^{(n+1)}(t,\xi) \right\vert &\leq \left\vert \mathcal{F}_0(\xi) e^{-(1+i(a\cdot\xi))t} \right\vert + \left\vert \int_0^t\int_{\mathbb{S}^{d-1}} b \left( \left\vert \frac{\xi}{\vert \xi \vert}\cdot \sigma \right\vert \right) \mathcal{F^{(n)}}(s,\overline{\xi}) e^{-(1+i(a\cdot\xi))(t-s)} \dd \sigma \dd s \right\vert \nonumber\\
&\leq \vert \mathcal{F}_0(\xi) \vert e^{-t} + \int_0^t \int_{\mathbb{S}^{d-1}} b \left( \left\vert \frac{\xi}{\vert \xi \vert}\cdot \sigma \right\vert \right) \vert \mathcal{F}^{(n)}(s,\overline{\xi}) \vert e^{-(t-s)} \dd \sigma \dd s \nonumber\\
&\leq \mathcal{G}_0(\xi) e^{-t} + \int_0^t \int_{\sigma\in\mathbb{S}^{d-1}} b \left( \left\vert \frac{\xi}{\vert \xi \vert}\cdot \sigma \right\vert \right) \mathcal{G}(s,\overline{\xi}) e^{-(t-s)} \dd \sigma \dd s \leq \mathcal{G}(t,\xi),
\end{align}
where in the last line we have used the assumption that $\mathcal{G}$ is a super-solution. We deduce by recursion that \eqref{EQUATComparisonPropa} holds for any index $n \in \mathbb{N}$. Passing to the limit $n\rightarrow+\infty$ in the \eqref{EQUATComparisonPropa}, we obtain that the solution $\mathcal{F}$ satisfies 
\begin{align}
\left\vert \mathcal{F}(t,\xi) \right\vert \leq \mathcal{G}(t,\xi).
\end{align}
for any $t \geq 0$ and $\xi \in \mathbb{R}^d$, which concludes the proof.
\end{proof}

\subsection{Exhibiting a super-solution to the evolution equation of the Fourier transform}

Now, we determine super-solutions of \eqref{EQUATFouriIntegTime_Vers1} in the sense of Definition \ref{DEFINSuper-SoluEquatFouri}, that are large enough to estimate from above the solutions of \eqref{EQUATFouriIntegTime_Vers1} in $X$. 
To ensure in addition that the super-solutions have an interesting long-time behaviour, we will consider functions of the form $u_p(t,\xi) = \varphi(t) \vert \xi \vert^p$, as in \cite{BoNV020}.\\
Contrary to the situation in \cite{BoNV020}, we will see that $u_p$ decays and converges to zero as $t \rightarrow +\infty$, as soon as $p > 0$, which will enable us to deduce the long-time behaviour of the solutions of \eqref{EQUATFouriIntegTime_Vers1}.

\begin{prop}[Super-solutions of \eqref{EQUATFouriIntegTime_Vers1}]
\label{PROPO_u_p_SuperSolut}
Let $p \in \mathbb{R}$ be any real number. Then, the function:
\begin{align}
u_p(t,\xi)= e^{(\lambda(p)-1)t} \vert \xi \vert^p,
\end{align}
where
\begin{align}
\label{EQUATDefinConstLambda_p__}
\lambda(p) = \int_{\mathbb{S}^{d-1}} b(\vert \sigma_1 \vert) \left( \frac{1+r^2}{2} + \frac{1-r^2}{2}\sigma_1\right)^{p/2} \dd \sigma,
\end{align}
and where $\sigma_1 = \sigma\cdot e_1$ is the first component of the vector $\sigma$, is a super-solution of \eqref{EQUATFouriIntegTime_Vers1} with initial datum $\vert \xi \vert^p$, in the sense of Definition \ref{DEFINSuper-SoluEquatFouri}.
\end{prop}

\begin{proof}
The key observation is that the functions of the form $C \vert \xi \vert^p$ are eigenfunctions of the gain part $Q^+$ of the collision operator defined in \eqref{eq:gainQ^+}. In fact, we have
\begin{align}
\int_{\mathbb{S}^{d-1}} b \left( \left\vert \frac{\xi}{\vert \xi \vert}\cdot \sigma \right\vert \right) \left\vert \frac{1-r}{2}\xi + \frac{1+r}{2}\vert \xi \vert \sigma \right\vert^p \dd \sigma = \vert \xi \vert^p \int_{\mathbb{S}^{d-1}} b \left( \left\vert \frac{\xi}{\vert \xi \vert}\cdot \sigma \right\vert \right) \left\vert \frac{1-r}{2} \frac{\xi}{\vert \xi \vert} + \frac{1+r}{2} \sigma \right\vert^p \dd \sigma
\end{align}
Considering the rotation $R$ that sends $e_1$ to $\xi/ \vert \xi \vert$, and the change of variables $\sigma = R(\sigma')$, we get
\begin{align}
\int_{\mathbb{S}^{d-1}} b \left( \left\vert \frac{\xi}{\vert \xi \vert}\cdot \sigma \right\vert \right) \left\vert \frac{1-r}{2}\xi + \frac{1+r}{2}\vert \xi \vert \sigma \right\vert^p \dd \sigma & = \vert \xi \vert^p \int_{ \mathbb{S}^{d-1}} b \left( \vert (e_1 \cdot \sigma') \vert \right) \left\vert \frac{1-r}{2} R(e_1) + \frac{1+r}{2} R(\sigma') \right\vert^p \dd \sigma' \nonumber\\
&= \vert \xi \vert^p \int_{\mathbb{S}^{d-1}} b \left( \vert (e_1 \cdot \sigma') \vert \right) \left\vert \frac{1-r}{2} e_1 + \frac{1+r}{2} \sigma' \right\vert^p \dd \sigma'.
\end{align}
Furthermore we compute
\begin{align}
\left\vert \frac{1-r}{2} e_1 + \frac{1+r}{2} \sigma' \right\vert^p &= \left( \left(\frac{1-r}{2}\right)^2 + \frac{(1-r)(1+r)}{2} \sigma'\cdot e_1 + \left(\frac{1+r}{2}\right)^2 \right)^{p/2} \nonumber\\
&= \left( \frac{1+r^2}{2} + \frac{1-r^2}{2} \sigma'_1 \right)^{p/2},
\end{align}
hence $Q^+\left[ \vert \xi \vert^p \right] = \lambda(p) \vert \xi \vert^p$ where $\lambda(p)$ is defined in \eqref{EQUATDefinConstLambda_p__}, and $Q^+$ denotes the gain term of the collision operator of \eqref{EQUATFouriDifferentiVersi}.\\
To verify that $u_p$ is a super-solution, we compute the right hand side of \eqref{EQUATDefinSuper-SoluEquatFouri}, with $\mathcal{G}_0(\xi) = \vert \xi \vert^p = u_p(0,\xi)$. Hence we have
\begin{align}
u_p(0,\xi) e^{-t}& + \int_0^t \int_{\sigma\in\mathbb{S}^{d-1}} b \left( \left\vert \frac{\xi}{\vert \xi \vert}\cdot\omega \right\vert \right) u_p(s,\overline{\xi})e^{-(t-s)} \dd \sigma \dd s \nonumber\\
&= e^{-t}\vert \xi \vert^p + \int_0^t e^{(\lambda(p)-1)s}e^{-(t-s)} Q^+\left[ \vert \xi \vert^p \right] \dd s \nonumber\\
&= e^{-t}\vert \xi \vert^p + e^{-t} \left( \int_0^t \lambda(p) e^{\lambda(p)s} \dd s \right) \vert \xi \vert^p \nonumber\\
&= e^{-t} \vert \xi \vert^p + e^{-t} \left[ e^{\lambda(p) t} - 1 \right] \vert \xi \vert^p \nonumber\\
&= u_p(t,\xi)
\end{align}
which proves that $u_p$ is a super-solution in the sense of Definition \ref{DEFINSuper-SoluEquatFouri}. 
\end{proof}
\noindent
To have complete understanding of $u_p$ it only remains to study the constant $\lambda(p)$, in particular it remains to compare it with $1$. The following result provides the description of $\lambda(p)$ we need.

\begin{prop}[Estimation of $\lambda(p)$]\label{PROPOCompaLambd__1__}
Let $p > 0$. Then, the constant $\lambda(p)$, defined in \eqref{EQUATDefinConstLambda_p__}, satisfies
\begin{align}
\lambda(p) < 1.
\end{align}
\end{prop}

\begin{proof}
First, we observe that for $p=0$, we have
\begin{align}
\lambda(0)= \int_{\mathbb{S}^{d-1}} b(\vert \sigma_1 \vert) \dd \sigma = 1.
\end{align}
We consider now the derivative of $\lambda$ with respect to $p$, i.e.
\begin{align}
\frac{\dd}{\dd p} \lambda(p) &= \frac{\dd}{\dd p} \int_{\mathbb{S}^{d-1}} b(\vert \sigma_1 \vert) e^{\frac{p}{2} \ln \left( \frac{1+r^2}{2} + \frac{1-r^2}{2}\sigma_1\right) } \dd \sigma \nonumber\\
&= \int_{\mathbb{S}^{d-1}} b(\vert \sigma_1 \vert) \frac{\ln \left( \frac{1+r^2}{2} + \frac{1-r^2}{2}\sigma_1\right)}{2} e^{\frac{p}{2} \ln \left( \frac{1+r^2}{2} + \frac{1-r^2}{2}\sigma_1\right) } \dd \sigma \nonumber\\
&= \frac{1}{2} \int_{\mathbb{S}^{d-1}} b(\vert \sigma_1 \vert) \ln \left( \frac{1+r^2}{2} + \frac{1-r^2}{2}\sigma_1\right) \left( \frac{1+r^2}{2} + \frac{1-r^2}{2}\sigma_1\right)^{p/2} \dd \sigma.
\end{align}
We observe now that $-1 < \sigma_1 < 1$ for almost every $\sigma \in \mathbb{S}^{d-1}$, then for almost every $\sigma$ we have
\begin{align}
\label{EQUATEstimIntegrand_Lambd_(p)_}
\frac{1+r^2}{2} + \frac{1-r^2}{2} \sigma_1 < \frac{1+r^2}{2} + \frac{1-r^2}{2} = 1,
\end{align}
which implies $\displaystyle{\frac{\dd}{\dd p} \lambda(p) < 0}$. Therefore, we get that $\lambda(p) < 1$, for all $p > 0$. 
\end{proof}

\begin{remark}
Observe that in the proof of Proposition \ref{PROPOCompaLambd__1__}, we used in a crucial manner that $r \neq 1$, in \eqref{EQUATEstimIntegrand_Lambd_(p)_}. Indeed, in the elastic case $r = 1$, $\lambda(p)$ is equal to the integral of the collision kernel $b$, and is therefore independent from $p$. In such a case, the super-solution $u_p$ is never vanishing for large $t$.
\end{remark}

\subsection{Conclusion of the long-time behaviour argument}

We are now ready to prove the stability of the steady state $\mathcal{F}_\infty$. We begin by stating the result main result of this subsection, namely the long-time behaviour for $\mathcal{F}$.

\begin{thm}[Long time behaviour of the solutions of \eqref{EQUATFouriIntegTime_Vers1}]
\label{THEORLong_Time_BehavSolutFouri}
Let $\mathcal{F}_0$, $\mathcal{G}_0$ be the Fourier transforms of two probability measures $f_0,g_0 \in \mathcal{M}_+(\mathbb{R}^d)$ respectively. 
Assume in addition that $f_0,g_0$ have finite first order moments. In other words, we assume:
\begin{align}
&\hspace{12mm}\int_{\R} f_0(\dd v) = \int_{\R} g_0(\dd v) = 1, \notag \\
&\int_{\R} |v|f_0(\dd v) < +\infty, \hspace{3mm} \int_{\R} |v|g_0(\dd v) < +\infty.
\end{align}
Let $\mathcal{F}$, $\mathcal{G}$ be the two solutions of \eqref{EQUATFouriIntegTime_Vers1} given by Proposition \ref{PROPOGlobaWell-PosedEquatFouriTrnsf} with respective initial data $\mathcal{F}_0$ and $\mathcal{G}_0$. Then $\mathcal{F}$ converges to $\mathcal{G}$ locally as $t \rightarrow + \infty$, i.e. for any compact set $K \subseteq \mathbb{R}^d$ we have
\begin{align}
\sup_{\xi \in \mathbb{R}^d} \left\vert \left( \mathcal{F}(t,\xi) - \mathcal{G}(t,\xi) \right) \mathds{1}_K(\xi) \right\vert \underset{t \mapsto +\infty }{\longrightarrow} 0.
\end{align}
In particular, if $\mathcal{F}_\infty$ denotes the Fourier transform of the steady state $f_\infty$ obtained in Theorem \ref{thm:steadystate}, we have 
\begin{align}
\sup_{\xi \in \mathbb{R}^d} \left\vert \left( \mathcal{F}(t,\xi) - \mathcal{F}_\infty(\xi) \right) \mathds{1}_K(\xi) \right\vert \underset{t \mapsto +\infty }{\longrightarrow} 0.
\end{align}
\end{thm}

\begin{proof}
By our assumption on $f_0,g_0$ we can apply Proposition \ref{PROPOCompaFouriFinitFirstMomnt}, which gives that 
\begin{align}
\left\vert \mathcal{F}_0(\xi) - \mathcal{G}_0(\xi) \right\vert \leq C \vert \xi \vert,
\end{align}
for any $\xi \in \mathbb{R}^d$. Consider now the solutions $\mathcal{F}$ and $\mathcal{G}$ of \eqref{EQUATFouriIntegTime_Vers1} associated to the initial data $\mathcal{F}_0$ and $\mathcal{G}_0$ respectively. By linearity and from Proposition \ref{PROPOGlobaWell-PosedEquatFouriTrnsf}, $\mathcal{F}-\mathcal{G}$ is the solution of \eqref{EQUATFouriIntegTime_Vers1} with initial datum $\mathcal{F}_0-\mathcal{G}_0$. Taking $p=1$ we observe that, from Proposition \ref{PROPOCompaFouriFinitFirstMomnt}, the difference $\mathcal{F}-\mathcal{G}$ is initially bounded from above by the function $C u_1(t,\xi)$, which, according to Proposition \ref{PROPO_u_p_SuperSolut}, is a super-solution of \eqref{EQUATFouriIntegTime_Vers1}. Therefore, by Proposition \ref{PROPOCompaPrincSuperSolutFouri}, we deduce that, for all $t \geq 0$ and $\xi \in \mathbb{R}^d$:
\begin{align}
\left\vert \mathcal{F}(t,\xi) - \mathcal{G}(t,\xi) \right\vert \leq Cu_1(t,\xi).
\end{align}
Finally from Proposition \ref{PROPOCompaLambd__1__} and the explicit expression of $u_1$, we conclude that $u_1$ converges to zero 
on every compact set $K \subseteq \mathbb{R}^d$ as $t \rightarrow +\infty$. 
\end{proof}

\noindent
Using now classical results (see for instance \cite{Fe2}) we can deduce, from the pointwise convergence of the Fourier transforms $\mathcal{F}(t,\cdot)$ of the solutions $f(t)$ to \eqref{eq:LinInBolstrong_Cauchy}, the weak convergence of the measures $\big( f(t) \big)_t$. 

\begin{cor}[Long time behaviour of the solution $f(t)$ of \eqref{eq:LinInBolstrong_Cauchy}]
\label{COROLLong_Time_Behav}
Let $f_0,g_0 \in \mathscr{M}(\R)$ be two probability measures with finite first order moments: 
\begin{align}
&\hspace{12mm}\int_{\R} f_0(\dd v) = \int_{\R} g_0(\dd v); \notag \\
&\int_{\R} |v|f_0(\dd v) < +\infty, \hspace{3mm} \int_{\R} |v|g_0(\dd v) < +\infty.
\end{align} 
Then, the two solutions $f,g$ of \eqref{eq:LinInBolstrong_Cauchy} given by Theorem \ref{thm:wp1} with respective initial data $f_0$ and $g_0$, are weakly converging to $0$ as $t\rightarrow +\infty$.
\end{cor}

\noindent
Applying in particular Corollary \ref{COROLLong_Time_Behav} to the case when one of the two functions $f,g$ is the steady state of the linear inelastic Boltzmann equation \eqref{EQUATIntroLineaInelaBoltz} that was obtained in Theorem \ref{thm:steadystate}, the proof of Theorem \ref{thm:stability} is complete.

\section{Long-time behaviour of the moments for general homogeneity $\mu \geq 0$ }\label{sec:Generalmoments}

In this section we will consider the evolution, and in particular the long-time behaviour, of the moments of the homogeneized version \eqref{EQUATIntroLineaInelaBoltzRehom} of \eqref{EQUATIntroLineaInelaBoltz}. We recall that the rehomogeneized homogeneous, inelastic linear Boltzmann equation, for Maxwell molecules, with a gravity field, is
\begin{align*}
\partial_t f(t,v) &+ a \cdot \partial_v f(t,v) = T^\mu(t) \Bigg[ \int_{\mathbb{S}^{d-1}} \frac{1}{r}  b(|\hspace{0.25mm}'N\hspace{-0.5mm} \cdot \omega|) f(t,'\hspace{-1mm}v) \dd \omega - f(t,v) \Bigg],
\end{align*}
where $T=T(t)$ is the temperature, $\mu \geq 0$ and, without loss of generality, we have assumed in this section that $\int_{\mathbb{S}^{d-1}} b\Big( \vert N \cdot \omega \vert\Big) \dd \omega = 1$.\\
As in Section \ref{sec:momentEst}, it is possible to derive the evolution equation for the moments of a solution of \eqref{EQUATIntroLineaInelaBoltzRehom}. Considering the weak formulation of the equation, and choosing $\{1,v,|v|^2\}$ as test functions we obtain the evolution equations of the moments of order zero, one and two. These evolution equations are explicitly written in the following proposition.

\begin{prop}\label{prop:momentsRenorm}
    Let $d=2,3$ and let the restitution coefficient $r$ be in $]0,1[$. Let $\mu > 0$ be a strictly positive real number. Let the angular collision kernel $b$ satisfy Assumption \ref{ass:kernelB} and suppose that $f$ is a solution of \eqref{EQUATIntroLineaInelaBoltzRehom}. Then we have
    \begin{align}
        \frac{\dd M_0}{\dd t} & = 0, \label{eq:evMassRenorm} \\
        \frac{\dd M_1}{\dd t} & = a_i M_0 - C_d(1+r) T^\mu M_{1}, \label{eq:evBulkVrenorm} \\
        \frac{\dd T}{\dd t} \hspace{1mm} & =  a \cdot M_1 - \frac{C_d}{2} (1-r^2) T^{\mu+1}, \label{eq:evTempRenorm}
    \end{align}
    where $C_d$ is a constant depending only on $d$ for which we have
    \begin{equation}
    \label{EQUATConst_C_d_}
        \begin{cases}
            C_d= 2\int_{-1}^1 b(x) \frac{x^2}{\sqrt{1-x^2}}\dd x & \text{for $d=2$,}\\
            C_d= 2 \pi\int_{-1}^1b(x)x^2 \dd x & \text{for $d=3$.}
        \end{cases}
    \end{equation}
\end{prop}
\noindent
The proof of the previous proposition, which we do not present here, follows the same lines of the ones provided in Section \ref{sec:momentEst}.

\subsection{Considerations on the phase space for the ODE system of the moments.}

The system of equations obtained in Proposition \ref{prop:momentsRenorm} is closed. However, the non-linear structure of the system in the case $\mu > 0$ prevents to solve it explicitly. 
Nevertheless, it is possible to obtain detailed information on the behaviour of the solutions to \eqref{eq:evBulkVrenorm}-\eqref{eq:evTempRenorm} with a qualitative study of these ODEs. 
Notice that independently on $\mu$ we always have $M_0(t)=M(0)$. Since the case $\mu = 0$ was already addressed in Section \ref{sec:momentEst}, we will focus on the case $\mu \neq 0$.\\
\newline
In order to study completely the evolution of the temperature, it is enough to understand the evolution of $a\cdot M_1$, the projection of the first moment along the direction of the acceleration field. We will denote such a projection by $x$:
\begin{align}
x(t) = a\cdot M_1(t).
\end{align}
Denoting the temperature by $y$, we will obtain a two-dimensional system of ODEs of the form $\frac{\dd}{\dd t} (x,y) = \big(f(x,y),g(x,y)\big)$ with $f,g:\mathbb{R}^2 \rightarrow \mathbb{R}$. Namely, from \eqref{eq:evBulkVrenorm}-\eqref{eq:evTempRenorm}, we deduce the following system of evolution equations for $x(t)$ and $y(t)$:
\begin{align}
\label{EQUATSyste2dMmt_Cas_MaxwlRehomGammaGener}
\left\{
\begin{array}{rcl}
\vspace{2mm}
\displaystyle{\frac{\dd}{\dd t}} x &=& \vert a \vert^2 M_0 - C_d(1+r) x y^\mu,\\
\displaystyle{\frac{\dd}{\dd t}} y &=& 2 x - C_d(1-r^2) y^{\mu+1}.
\end{array}
\right.
\end{align}
Since the temperature $T$ is defined as
\begin{align}
T(t) = y(t) = \frac{1}{2}\int_{\R} \vert v \vert^2 f(t,v) \dd v,
\end{align}
such a quantity can never be negative. On the other hand, the projection $x$ of the first moment can have both signs. We will therefore work in the phase space $y \geq 0$. In particular, it is important to ensure that the solutions $t \mapsto \left(x(t),y(t)\right)$ never lead to unphysical situations where $y(t) < 0$.\\
As a matter of fact, by considering some particular points $(x,y)$ of the phase space of the form $(x,0)$ (namely, choosing $x < 0$), it is possible to show that the solution \eqref{EQUATSyste2dMmt_Cas_MaxwlRehomGammaGener} to the moment system through such points are such that $y' < 0$, leading therefore to a negative temperature for arbitrarily small positive times.\\
In addition, the first and second moments are not independent, because they are both linked to the distribution function $f$. Indeed, we have
\begin{align}
\label{EQUATLink_FirstSeconMmnts}
\vert x(t) \vert &= \left\vert \int_{\R} (a\cdot v) f(t,v) \dd v \right\vert \leq \vert a \vert \int_{\R} \vert v \vert f(t,v) \dd v \leq \vert a \vert \sqrt{\int_{\R} f(t,v) \dd v} \cdot \sqrt{\int_{\R} \vert v \vert^2 f(t,v) \dd v} \nonumber\\
&\leq \vert a \vert \sqrt{M_0} \sqrt{y(t)}.
\end{align}
As a consequence, the relevant region in the phase space of the ODE system on $(x,y)$ is the hypergraph of the function $\displaystyle{y = \frac{x^2}{\vert a \vert M_0}}$. We observe in particular that points of the form $(x,0)$, with $x < 0$, do not belong to such a region, ruling out the only possibility to produce solutions with a temperature switching from positive and physical values, to negative values.\\
We can prove the following result, that can be seen as a test of consistency for the system of the moments \eqref{EQUATSyste2dMmt_Cas_MaxwlRehomGammaGener}.

\begin{prop}
\label{PROPOInvarRegioPhaseSpace}
Let $a \in \mathbb{R}^d$ be a fixed vector, and let $C_d$ be as in \eqref{EQUATConst_C_d_}. Let $\mu > 0$ be a strictly positive real number. Then, the region:
\vspace{-5mm}
\begin{align}
\label{EQUATPhysiRegioPhaseSpace}
\mathcal{R} = \{(x,y) \in \mathbb{R}^2\ /\ \vert a \vert M_0 y \geq x^2\}
\end{align}
is an invariant region under the flow of the ODE system \eqref{EQUATSyste2dMmt_Cas_MaxwlRehomGammaGener}.
\end{prop}

\begin{proof}
At any point of coordinates $(x,y)$ of the boundary $x^2 = \vert a \vert^2 M_0 y$ of the region $\mathcal{R}$, a normal, ingoing vector $n(x,y)$ to the boundary at this point is given by
\begin{align}
n(x,y) = \left( -2x,\vert a \vert^2 M_0 \right).
\end{align}
Computing the scalar product between the ingoing normal vector $n$ and the time derivative of the integral curve of the system \eqref{EQUATSyste2dMmt_Cas_MaxwlRehomGammaGener} through the point $(x,y)$ provides
\begin{align}
n(x,y) \cdot \left(x'(t),y'(t)\right) &= -2x\left( \vert a \vert^2 M_0 - C_d(1+r)xy^\mu \right) + \vert a \vert^2 M_0 \left( 2x - C_d(1-r^2)y^{\mu+1} \right) \nonumber\\
&= C_d \vert a \vert^2 M_0 y^{\mu +1} (1+r)^2.
\end{align}
We observe, on the one hand, that we can always assume $M_0 > 0$, and on the other hand that $a > 0$. In the case $a = 0$, that is, in the case of a Lorentz gas of Maxwell molecules, without any acceleration field, the projection $x$ we introduced is zero by definition, and the evolution equation of the temperature becomes the simple one-dimensional ODE:
\begin{align}
\frac{\dd}{\dd t} y = - C_d (1-r^2) y^{\mu +1},
\end{align}
which implies that a system starting from any positive temperature will have a positive temperature for any time (in the future, but also in the past).\\
Therefore, in the ``non trivial'' case $a \neq 0$, $M_0 \neq 0$, we find that $n(x,y) \cdot (x',y') > 0$, proving that the hypergraph of $x^2 = \vert a \vert^2 M_0 y$ is an invariant region under the flow of the ODE system \eqref{EQUATSyste2dMmt_Cas_MaxwlRehomGammaGener}. In other words, no solution of the ODE system starting in the hypergraph of $x^2 = \vert a \vert^2 M_0 y$ can leave this region.
\end{proof}

\begin{remark}
The result of Proposition \ref{PROPOInvarRegioPhaseSpace} can be interpreted as follows. Any physical choice of the initial data for the variables $(x_0,y_0)$ (in the sense that the initial data satisfy the condition \eqref{EQUATLink_FirstSeconMmnts}) lead to an integral curve that will satisfy the same physical condition \eqref{EQUATLink_FirstSeconMmnts} for all positive times.
\end{remark}

\noindent
To conclude this preliminary discussion about the ODE system \eqref{EQUATSyste2dMmt_Cas_MaxwlRehomGammaGener}, let us observe that the only equilibrium $\left(x_\infty,y_\infty\right)$ of the system belongs to the invariant region $\vert a \vert^2 M_0 y \geq x^2$. Indeed, by assumption, we have:
\begin{align}
\vert a \vert^2 M_0 = C_d(1+r) x_\infty y_\infty^\mu \hspace{3mm} \text{and} \hspace{3mm} 2x_\infty = C_d(1-r^2) y_\infty^{\mu+1}.
\end{align}
Therefore, we have $2x_\infty^2 = C_d(1-r^2) x_\infty y_\infty^\mu \cdot y_\infty = C_d(1-r^2) \displaystyle{\frac{\vert a \vert^2 M_0}{C_d(1+r)}} y_\infty = (1-r) \vert a \vert^2 M_0 y_\infty$, so that
\begin{align}
x_\infty^2 < \vert a \vert^2 M_0 y_\infty.
\end{align}

\begin{remark}
We found that the parabola corresponding to the boundary of the region $\mathcal{R}$, and which describes the consistency relation between the first and second moment coming from the Cauchy-Schwarz inequality, cannot be crossed by solutions of the ODE system \eqref{EQUATSyste2dMmt_Cas_MaxwlRehomGammaGener} evolving forward.\\
Nevertheless, it is possible to consider solutions of \eqref{EQUATSyste2dMmt_Cas_MaxwlRehomGammaGener} starting exactly on such a parabola. For positive times, the solutions will remain always inside the hypergraph of the parabola, but for negative times, the solutions will leave this region, leading to solutions that cannot be associated to any distribution function $f$ solving the inelastic Lorentz equation. Conversely, one can consider the solution of the ODE system \eqref{EQUATSyste2dMmt_Cas_MaxwlRehomGammaGener} starting inside the region $\mathcal{R}$, for times both positive and negative. By doing so, we would observe that the solutions leave (for negative times) the region $\mathcal{R}$. This strongly suggests that \eqref{EQUATIntroLineaInelaBoltzRehom} is not well-posed globally for negative times, and that the solutions of this equation develop singularity in finite, negative times.
\end{remark}

\subsection{Long time behaviour of the moments}\label{par:gamma>0}

In this section, we will carry out the study of the qualitative behaviour of the solutions of the system \eqref{EQUATSyste2dMmt_Cas_MaxwlRehomGammaGener}. In particular, we will prove that all the solutions (starting inside the physically relevant region $\mathcal{R}$) converge to the only equilibrium of the system.\\
First, we establish that all the solutions of \eqref{EQUATSyste2dMmt_Cas_MaxwlRehomGammaGener} eventually enter the first quadrant in finite time.

\begin{prop}[Invariance and attractivity of the first quadrant]
\label{PROPOInvarRegioPhaseSpaceQuadr}
Let $a \in \mathbb{R}^d$ be a fixed vector, and let $C_d$ be as in \eqref{EQUATConst_C_d_}. Let $\mu > 0$ be a strictly positive real number.\\
Then, the region:
\begin{align}
\label{EQUATPhysiRegioPhaseSpaceQuadr}
\mathcal{R} \cap \{(x,y)\in\mathbb{R}^2\ /\ x \geq 0, y \geq 0\},
\end{align}
where $\mathcal{R}$ is defined in \eqref{EQUATPhysiRegioPhaseSpace}, is an invariant region under the flow of the ODE system \eqref{EQUATSyste2dMmt_Cas_MaxwlRehomGammaGener}. In addition, any solution of the system starting initially from a point of $\mathcal{R}$ enters the first quadrant $\{(x,y)\in\mathbb{R}^2\ /\ x \geq 0, y \geq 0\}$ in finite time.
\end{prop}

\begin{proof}
We start to prove that the first quadrant $\{(x,y)\in\mathbb{R}^2\ /\ x > 0, y > 0\}$ is an invariant region under the flow of the ODE system \eqref{EQUATSyste2dMmt_Cas_MaxwlRehomGammaGener}.\\
If a solution $(x(t),y(t))$ satisfies $x(t_0) = 0$ for some time $t_0$, then $x'(t_0) > 0$, so such a solution would be outside at $t_0-h$ for $h>0$ small enough. We obtain therefore a contradiction by considering a solution starting inside $\mathcal{R}$ and by considering for $t_0$ the first time this solution leaves the first quadrant. 
If $y(t_0) = 0$ and $x(t_0) > 0$ then $y'(t_0) > 0$, and the same argument applies. If finally $x(t_0) = y(t_0) = 0$, then the solution lies on the boundary of $\mathcal{R}$ at time $t_0$, which is an invariant region under the flow of \eqref{EQUATSyste2dMmt_Cas_MaxwlRehomGammaGener} by Proposition \ref{PROPOInvarRegioPhaseSpace}.\\
We now prove that all the solutions starting inside $\mathcal{R}$ enter eventually in $\mathcal{R} \cap \{(x,y)\in\mathbb{R}^2\ /\ x \geq 0, y \geq 0\}$. 
Let us assume that a solution $(x(t),y(t))$ of the system \eqref{EQUATSyste2dMmt_Cas_MaxwlRehomGammaGener} remains for all positive time in the quadrant $x \leq 0, y \geq 0$. Since this solution remains in $\mathcal{R}$ by Proposition \ref{PROPOInvarRegioPhaseSpace}, we observe that $x' \geq 0$, $y' \leq 0$, so that the coordinate $x$ being bounded from above and the coordinate $y$ being bounded from below, the solution is globally defined, and its integral curve should converge towards a limit, which has to be an equilibrium, but there is no such an equilibrium in $\mathcal{R}$ intersected with the quadrant $x \leq 0, y \geq 0$, so all the solutions enter the region $x \geq 0, y \geq 0$ in finite time. The proof is now complete.
\end{proof}

\noindent
Therefore, we need only to study the phase portrait of the ODE system \eqref{EQUATSyste2dMmt_Cas_MaxwlRehomGammaGener} in the first quadrant intersected with the invariant region $\mathcal{R}$, as pictured in Figure \ref{FIGURPhasePortrMomntMaxvlRehomGam>0}.

\begin{figure}[h]
    \centering
    \includegraphics[scale=0.3]{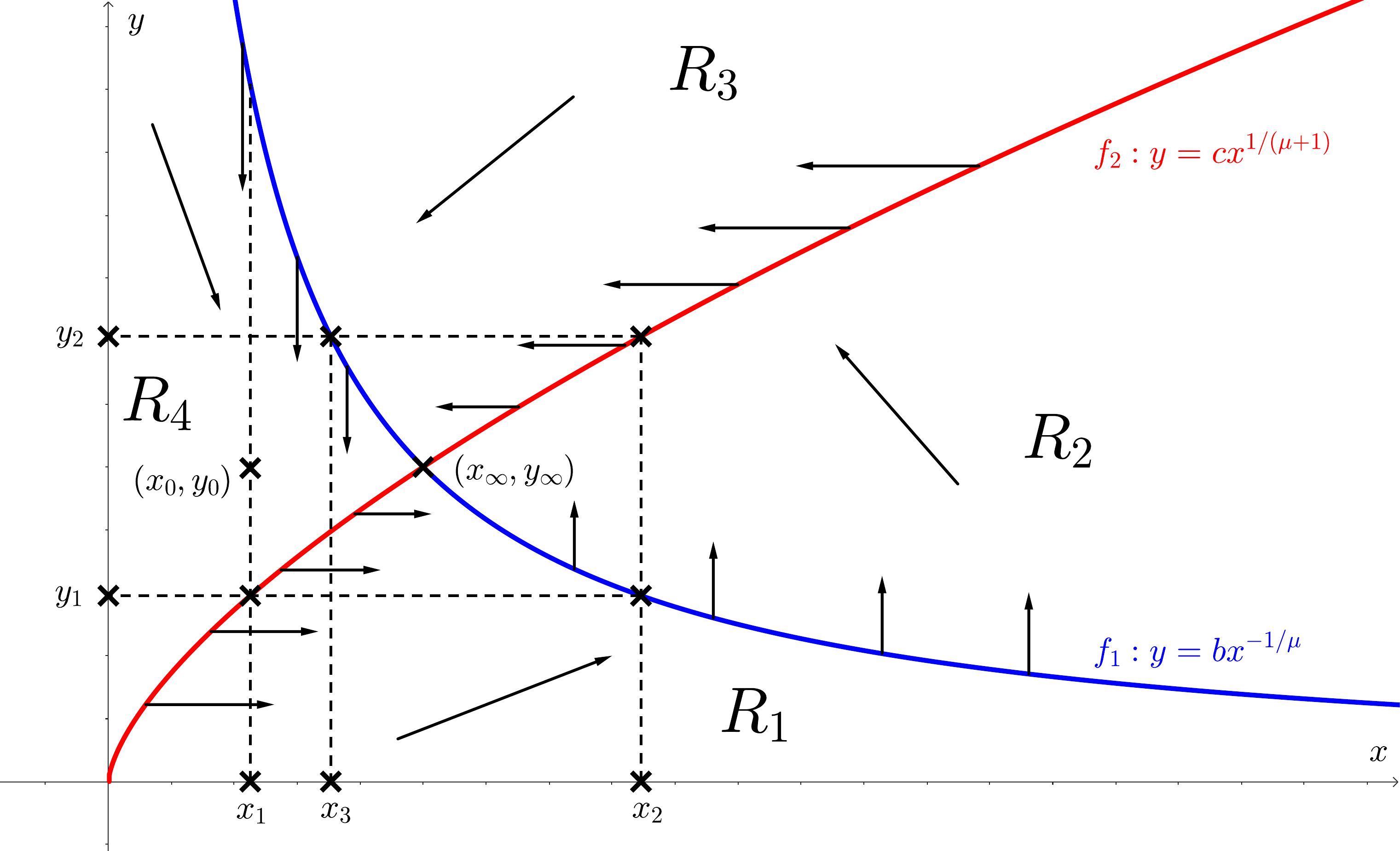}
    \caption{The phase portrait of the ODE system \eqref{EQUATSyste2dMmt_Cas_MaxwlRehomGammaGener}. The vertical isocline is represented in blue, the horizontal isocline is represented in red. At the intersection between the two isoclines lies the only equilibrium of the system \eqref{EQUATSyste2dMmt_Cas_MaxwlRehomGammaGener}.}
    \label{FIGURPhasePortrMomntMaxvlRehomGam>0}
\end{figure}

\noindent
We obtain the following result.

\begin{prop}
Let $a \in \mathbb{R}^d$ be a fixed vector, and let $C_d$ be as in \eqref{EQUATConst_C_d_}. Let $\mu > 0$ be a strictly positive real number.\\
Then, any solution of the ODE system \eqref{EQUATSyste2dMmt_Cas_MaxwlRehomGammaGener} that initially starts in the region $\mathcal{R}$, where $\mathcal{R}$ is defined in \eqref{EQUATPhysiRegioPhaseSpace}, converges towards the unique fixed point of the system \eqref{EQUATSyste2dMmt_Cas_MaxwlRehomGammaGener} contained in the region $\mathcal{R}$.
\end{prop}

\begin{proof}
To study the qualitative behaviour of the integral curves of \eqref{EQUATSyste2dMmt_Cas_MaxwlRehomGammaGener}, we determine first the vertical and horizontal isoclines, providing that a single equilibrium of the system exists in the hypergraph of the parabola $x^2 = \vert a \vert^2 M_0 y$, and it lies in the first quadrant.

\noindent
The vertical isocline, characterized by the points $(x,y)$ of the phase space for which the integral curve through this point satisfies $x' = 0$, is the curve (represented in blue in Figure \ref{FIGURPhasePortrMomntMaxvlRehomGam>0}) of equation 
\begin{align}
y = b x ^{-1/\mu} \hspace{3mm} \text{with} \hspace{3mm} b = \left(\frac{\vert a \vert^2 M_0}{C_d(1+r)}\right)^{1/\mu}.
\end{align}
In the same way, the horizontal isocline (such that $y'=0$) is the curve (represented in red in Figure \ref{FIGURPhasePortrMomntMaxvlRehomGam>0}) of equation
\begin{align}
y = c x^{1/(\mu+1)} \hspace{3mm} \text{with} \hspace{3mm} c = \left( \frac{2}{C_d (1-r^2)} \right)^{1/(\mu+1)}.
\end{align}
In addition to the fact that their intersection corresponds to the only equilibrium of \eqref{EQUATSyste2dMmt_Cas_MaxwlRehomGammaGener}, the two isoclines separate the phase space in four different regions $R_i$, $1 \leq i \leq 4$, labelled counter-clockwise (see Figure \ref{FIGURPhasePortrMomntMaxvlRehomGam>0}): $R_1$, below the two isoclines, $R_2$ below the horizontal isocline and above the vertical isocline, $R_3$ above the two isoclines, and finally $R_4$ above the horizontal isocline and below the vertical isocline.\\
\newline
We observe now that if a solution enters in one of the regions $R_i$, and remains in this region for any positive time, then the solution is global, and it converges towards the only equilibrium $(x_\infty,y_\infty)$. Let us present the arguments that allow to reach such a result, for the particular case of the region $R_1$. The cases of the three other regions can be studied in the exact same manner. So, we assume that we have a solution that remains in $R_1$ for any positive time of its interval of definition. In $R_1$, we have $x' \geq 0$ and $y' \geq 0$. In particular, we have $y(t) \geq y(0)$ for any $t \geq 0$ of the time interval of definition. The intersection between $R_1$ and $y \geq y(0)$ defines a bounded region of the phase space, we deduce therefore that such a solution has to be global. The coordinates $x$ and $y$ of the solution being bounded from above and increasing, they both converge towards finite limits as $t \rightarrow +\infty$. 
Such a limit has therefore to be an equilibrium, because if not, one of the derivatives $x'$ or $y'$ converges towards a non zero limit, which in turn contradicts the boundedness of the solution.\\
\newline
We have then proved, for any solution, that either such a solution eventually remains in one of the four regions $R_i$ forever, in which case the solution converges towards the equilibrium, or such a solution leaves any of the four regions $R_i$ it entered in finite time. It remains to study this second case in more details.\\
Let us consider a solution initially in $R_1$, that leaves this region. Since the first quadrant is an invariant region, the solution cannot leave $R_1$ through the first axis $y = 0$. The solution cannot neither leave $R_1$ through the equilibrium. Finally, let us observe that at the intersection between $R_1$ and $R_4$ (except at the equilibrium), the derivative of a solution at such a point satisfies $x' > 0$, so that no solution cannot leave $R_1$ by entering $R_4$. In the end, the only way a solution can leave $R_1$ is to enter $R_2$.\\
With the exact same arguments, we would obtain that the only way to leave $R_2$ is to enter $R_3$, the only way to leave $R_3$ is to enter $R_4$, and the only way to leave $R_4$ is to enter $R_1$. Therefore, a solution that leaves any of the regions $R_i$ in finite time travels the regions according to the cyclic order $R_1 \rightarrow R_2 \rightarrow R_3 \rightarrow R_4 \rightarrow R_1 \rightarrow \dots$, forever. 
Without loss of generality, we consider a solution initially starting from $\left(x_0,y_0\right) \in R_4$. We can now deduce the following.
\begin{itemize}
\item In $R_4$, $x' \geq 0$ and $y' \leq 0$. In particular, as long as the solution remains in $R_4$ we have $x(t) \geq x_0 = x_1$. The intersection between $R_4$ and $x \geq x_0$ is inside $y \geq y_1$ (see Figure \ref{FIGURPhasePortrMomntMaxvlRehomGam>0}), with:
\begin{align}
y_1 = c x_1^{1/(\mu+ 1)}.
\end{align}
\item The solution enters $R_1$ at a point $(x_{R_1},y_{R_1})$ such that $x_{R_1} \geq x_1$ and $y_{R_1} \geq y_1$. But now in $R_1$, we have $x' \geq 0$ and $y' \geq 0$. Therefore, in $R_1$ the solution remains in the region $y \geq y_{R_1} \geq y_1$. It implies that the solution remains in the region $x \leq x_2$, with
\begin{align}
y_1 = b x_2^{-1/\mu} \hspace{3mm} \text{that is} \hspace{3mm} x_2 = \left(y_1/b\right)^{-\mu}.
\end{align}
\item The solution enter then $R_2$ at a point $(x_{R_2},y_{R_2})$, with $x_{R_2} \leq x_2$. In $R_2$, since $x' \leq 0$, the solution remains in the region $x \leq x_2$. As a consequence, the solution remains in the region $y \leq y_2$, with
\begin{align}
y_2 = c x_2^{1/(\mu+1)}.
\end{align}
\item When the solution enters afterwards in $R_3$, it does at a point $(x_{R_3},y_{R_3})$ with $y_{R_3} \leq y_2$. In $R_3$ we have $y' \leq 0$, so that the solution remains in the region $y \leq y_2$. Therefore, the solution remains in the region $x \geq x_3$ before it leaves $R_3$, with
\begin{align}
y_2 = b x_3^{-1/\mu} \hspace{3mm} \text{that is} \hspace{3mm} x_3 = (y_2/b)^{-\mu}.
\end{align}
\end{itemize}
In the end, we obtained that when a solution starting from $(x_0,y_0)$ in $R_1$ leaves this region $R_1$, it re-enters this region $R_1$ for the first time at a point $(x_{R_1},y_{R_1})$ with
\begin{align}
x_{R_1} \geq x_3.
\end{align}
But by definition of the different values $x_i$, $y_i$, we have
\begin{align}
x_3 = \left( y_2/b \right)^{-\mu} = \left( \frac{c x_2^{1/(\mu+1)}}{b} \right)^{-\mu} = \left( \frac{c \left( y_1/b \right)^{-\mu/(\mu+1)}}{b} \right)^{-\mu} = \left( \frac{c \left( \displaystyle{\frac{ c x_1^{1/(\mu+1)}}{b}} \right)^{-\mu/(\mu+1)}}{b} \right)^{-\mu} \cdotp
\end{align}
Simplifying, we obtain
\begin{align}
x_3 = \frac{c^{-\mu} c^{\mu^2/(\mu+1)}}{b^{-\mu}b^{\mu^2/(\mu+1)}} x_1^{\mu^2/(\mu+1)^2} = \left( \frac{c}{b} \right)^{-\mu + \frac{\mu^2}{\mu + 1}} x_1^\frac{\mu^2}{(\mu+1)^2}.
\end{align}
We will now show that $x_3 > x_1$, assuming that $x_1 < x_\infty$, where $x_\infty$ is the abscissa of the equilibrium point $(x_\infty,y_\infty)$ of the system \eqref{EQUATSyste2dMmt_Cas_MaxwlRehomGammaGener}. Such a statement is equivalent to prove that
\begin{align}
\left( \frac{c}{b} \right)^{-\mu + \frac{\mu^2}{\mu + 1}} x_1^{\frac{\mu^2}{(\mu+1)^2}-1} > 1.
\end{align}
The first step to prove the result is to determine more explicitly the coordinates of the equilibrium point. By definition of this equilibrium, we have
\begin{align}
y_\infty = bx_\infty^{-1/\mu} \hspace{3mm} \text{and} \hspace{3mm} y_\infty = c x_\infty^{1/(\mu+1)},
\end{align}
so that
\begin{align}
\frac{b}{c} = x_\infty^{1/\mu}x_\infty^{1/(\mu+1)} = x_\infty^{\frac{2\mu+1}{\mu(\mu+1)}}.
\end{align}
By assumption, we considered an initial datum of the ODE system that belongs to the region $R_4$, so that we have $x_1 < x_\infty$. Observing that the number $\frac{\mu^2}{(\mu+1)^2}-1$ is negative, we see that the function $x \mapsto x^{\frac{\mu^2}{(\mu+1)^2}-1}$ is decreasing on $\mathbb{R}_+^*$, providing that
\begin{align}
x_1^{\frac{\mu^2}{(\mu+1)^2}-1} > x_\infty^{\frac{\mu^2}{(\mu+1)^2}-1}.
\end{align}
Computing now the right hand side of the last inequality we obtain
\begin{align}
x_\infty^{\frac{\mu^2}{(\mu+1)^2}-1} = \left( \left(\frac{b}{c}\right)^{\frac{\mu(\mu+1)}{2\mu+1}} \right)^{\frac{\mu^2}{(\mu+1)^2}-1}.
\end{align}
In the end, we will have $x_3 > x_1$ if
\begin{align}
\left( \frac{c}{b} \right)^{-\mu + \frac{\mu^2}{\mu + 1}} \left( \left(\frac{b}{c}\right)^{\frac{\mu(\mu+1)}{2\mu+1}} \right)^{\frac{\mu^2}{(\mu+1)^2}-1} \geq 1.
\end{align}
The power of $b/c$ in the previous inequality is
\begin{align}
&\mu - \frac{\mu^2}{\mu+1} + \frac{\mu(\mu+1)}{2\mu+1}\left(\frac{\mu^2}{(\mu+1)^2}-1\right) = \frac{\mu}{\mu+1} + \frac{\mu(\mu+1)}{2\mu+1}\left(\frac{\mu^2}{(\mu+1)^2}-1\right) \nonumber\\
&= \frac{\mu(2\mu+1) +\mu(\mu^2-(\mu+1)^2)}{(\mu+1)(2\mu+1)} \nonumber\\
&= \frac{2\mu^2 + \mu -2\mu^2 - \mu}{(\mu+1)(2\mu+1)} = 0.
\end{align}
Therefore, we have proved that, for $x_1 < x_\infty$, we have
\begin{align}
x_3 = \left( \frac{c}{b} \right)^{-\mu + \frac{\mu^2}{\mu+1}} x_1^{\frac{\mu^2}{(\mu+1)^2}-1} x_1 > \underbrace{\left( \frac{c}{b} \right)^{-\mu + \frac{\mu^2}{\mu+1}} x_\infty^{\frac{\mu^2}{(\mu+1)^2}-1}}_{=1} x_1 = x_1.
\end{align}
As a consequence, if we denote by $(x_{1,n})_{n\geq 0}$ the sequence of the abscissas of the consecutive re-entry points in $R_4$ of a given solution of \eqref{EQUATSyste2dMmt_Cas_MaxwlRehomGammaGener} that does not remain eventually in any of the four regions $R_i$, then we have proved that such a sequence is strictly increasing, bounded from above by $x_\infty$, hence converging, towards a certain limit $x_{1,\infty}$.\\
If we assume that $x_{1,\infty} < x_\infty$, we would reach a contradiction by considering a trajectory starting from a point with its abscissa equal to $x_{1,\infty}$: by continuity such a trajectory would leave $R_1$ and re-enter it, and the abscissa $\overline{x}$ of the point of re-entry would satisfy $x_{1,\infty} < \overline{x}$. Considering now two trajectories starting from points with respective abscissae $x_{1,n}$ and $x_{1,n+1}$, we would reach a contradiction by the continuity of the solutions of \eqref{EQUATSyste2dMmt_Cas_MaxwlRehomGammaGener} with respect to the initial data. We deduce that $x_{1,\infty} = x_\infty$, and therefore, any solution of \eqref{EQUATSyste2dMmt_Cas_MaxwlRehomGammaGener} that does not remain eventually in any of the four regions $R_i$ is global, and converges towards the only equilibrium of \eqref{EQUATSyste2dMmt_Cas_MaxwlRehomGammaGener} as $t \rightarrow +\infty$.
\end{proof}

\section{Conclusion}

In the present article, we studied the linear inelastic Boltzmann equation \eqref{EQUATIntroLineaInelaBoltz}, in the case of a Lorentz gas, in the space-homogeneous case, for Maxwell molecules and in the presence of a gravity field.\\
We obtained well-posedness results for such an equation, proving the existence and uniqueness of measure-valued solutions that solve the linear inelastic Boltzmann equation in the weak sense, for a given measure-valued initial datum. In addition, we proved that a unique steady state exists in the class of the non-negative Radon measures with finite first moment, and we proved that such a steady state is attracting all measure-valued solutions with a finite first moment, relying on the Fourier transform approach developed by A.~Bobylev.\\
We also provided a complete study of the system of moments associated to the rehomogeneized linear inelastic Boltzmann equation \eqref{EQUATIntroLineaInelaBoltzRehom} for Maxwell molecules.\\
\newline
It remains in future works to establish the same sort of well-posedness results for the linear inelastic Boltzmann equation, considering more general collision kernels. In particular, it is desirable to extend the results to the more physical case of the hard sphere collision kernel. Nevertheless, most of the methods developed in the present articles seem to fail in such a case: the semigroup approach is more difficult to apply, and the study of the system of the moments is notably more difficult outside the case of the Maxwell molecules.\\
It would be also interesting to solve the rehomogeneized linear inelastic Boltzmann equation \eqref{EQUATIntroLineaInelaBoltzRehom}. Indeed, although this equation cannot be considered as a completely relevant model from the physical point of view, due to the homogeneity of the collision operator \eqref{EQUATIntroLineaInelaBoltzRehom} the decay of the temperature of its solutions should reflect the behaviour of physical models.

\medskip

\begin{appendices}

\section{Evolution equation for the Fourier transform}\label{app:evolFourier}

Here, we write the equation satisfied by the Fourier transform $\mathcal{F}(t,\xi) = \int_{\R} f(t,v) e^{-i v\cdot \xi} \dd v$. Without loss of generality we will here assume that $\int_{\mathbb{S}^{d-1}}b(|N \cdot \omega|) \dd \omega=1$, which will allow us to simplify the computations. We start from the evolution equation \eqref{eq:LinInBolstrong_Cauchy} of $f$, and analyze the terms one by one. We have
\begin{align}
\partial_t \mathcal{F}(t,\xi) = - \int_{\R} a\cdot\partial_v f(t,v) e^{-i v\cdot\xi} \dd v + \int_{\R} \int_{\mathbb{S}^{d-1}} \frac{1}{r}b\left( \left\vert \frac{v}{\vert v \vert}\cdot\omega \right\vert \right) f(t,'\hspace{-1mm}v) e^{-iv\cdot\xi} \dd v - \mathcal{F}(t,\xi).
\end{align}
The first term on the right-hand side can be rewritten as
\begin{align}
-\int_{\R} a\cdot\partial_v f(t,v) e^{-i v\cdot\xi} \dd v = a\cdot\int_{\R} f(t,v) \partial_v e^{-iv\cdot \xi} \dd v = -i (a\cdot \xi) \mathcal{F}(t,\xi).
\end{align}
As far as the gain term is concerned, we start with applying the change of variable $'\hspace{-0.5mm}v \to v$ for $\omega$ fixed. This yields
\begin{align}
\int_{\R} \int_{\mathbb{S}^{d-1}} \frac{1}{r}b\left( \left\vert \frac{v}{\vert v \vert}\cdot\omega \right\vert \right) f(t,'\hspace{-1mm}v) e^{-iv\cdot\xi} \dd v = \int_{\R} \int_{\mathbb{S}^{d-1}} b\left( \left\vert \frac{v}{\vert v \vert}\cdot\omega \right\vert \right) f(t,v) e^{-iv'\cdot\xi} \dd v.
\end{align}
Then, we replace the $\omega$-representation by the $\sigma$-representation. We introduce the collision kernel in $\sigma$-representation $b_\sigma$, defined as
\begin{align}
b_\sigma\left( \left\vert \frac{v}{\vert v \vert} \cdot \sigma \right\vert \right) = b \left( \left\vert \frac{v}{\vert v \vert}\cdot\omega \right\vert\right)
\end{align}
for $\sigma = \frac{v}{\vert v \vert} - 2 \left( \frac{v}{\vert v \vert}\cdot\omega \right) \omega$, and we perform the change of variables $\omega \rightarrow \sigma$. Note that after this change of variables, $v'$ reads
\begin{align}
v'=v'(\sigma) = (1-r) \frac{v}{2} + (1+r) \frac{\vert v \vert}{2} \sigma.
\end{align}
To obtain a closed equation on the Fourier transform $\mathcal{F}$, we use the inversion formula 
from which the second term becomes
\begin{align}
\int_{\R}\int_{\R}\int_{\mathbb{S}^{d-1}} b_\sigma \left( \left\vert \frac{v}{\vert v \vert} \cdot \sigma \right\vert\right) \frac{\mathcal{F}(t,w)}{(2\pi)^d} e^{i v \cdot w} e^{- i \left[ \frac{1-r}{2} v + \frac{1+r}{2} \vert v \vert \sigma \right] \cdot \xi} \dd \sigma \dd w \dd v.
\end{align}
In the last integral, for $v$ and $w$ fixed, we introduce the orthogonal symmetry $R = R_{v,\xi}$, which maps $\frac{v}{\vert v \vert}$ to $\frac{\xi}{\vert \xi \vert}$, and we consider the change of variables $\sigma = R(\widetilde{\sigma})$. Using the fact that $R$ is an involution, we get that $R$ also sends  $\frac{\xi}{\vert \xi \vert}$ to $\frac{v}{\vert v \vert}$. Hence we have that
\begin{align}
\int_{\R}\int_{\R}\int_{\mathbb{S}^{d-1}} &b_\sigma \left( \left\vert \frac{v}{\vert v \vert} \cdot \sigma \right\vert\right) \frac{\mathcal{F}(t,w)}{(2\pi)^d} e^{i v \cdot w} e^{- i \left[ \frac{1-r}{2} v + \frac{1+r}{2} \vert v \vert \sigma \right] \cdot \xi} \dd \sigma \dd w \dd v \nonumber\\
&= \int_{\R} \int_{\R} \int_{\mathbb{S}^{d-1}} b_\sigma \left( \left\vert R\left(\frac{\xi}{\vert \xi \vert} \right)\cdot \sigma \right\vert\right) \frac{\mathcal{F}(t,w)}{(2\pi)^d} e^{i v \cdot w} e^{- i \left[ \frac{1-r}{2} v \cdot \xi + \frac{1+r}{2} \vert v \vert \vert \xi \vert \sigma \cdot  R \left( \frac{v}{\vert v \vert} \right) \right]} \dd \sigma \dd w \dd v \nonumber\\
&= \int_{\R} \int_{\R} \int_{\mathbb{S}^{d-1}} b_\sigma \left( \left\vert R\left(\frac{\xi}{\vert \xi \vert} \right)\cdot R(\widetilde{\sigma}) \right\vert\right) \frac{\mathcal{F}(t,w)}{(2\pi)^d} e^{i v \cdot w} e^{- i \left[ \frac{1-r}{2} v\cdot \xi + \frac{1+r}{2} \vert v \vert \vert \xi \vert R(\widetilde{\sigma}) \cdot  R \left( \frac{v}{\vert v \vert} \right) \right]} \dd \widetilde{\sigma} \dd w \dd v \nonumber\\
&= \int_{\R} \int_{\R} \int_{\mathbb{S}^{d-1}} b_\sigma \left( \left\vert \frac{\xi}{\vert \xi \vert} \cdot \widetilde{\sigma} \right\vert\right) \frac{\mathcal{F}(t,w)}{(2\pi)^d} e^{i v \cdot w} e^{- i \left[ \frac{1-r}{2} v\cdot\xi + \frac{1+r}{2} \vert \xi \vert \widetilde{\sigma} \cdot  v \right]} \dd \widetilde{\sigma} \dd w \dd v.
\end{align}
Rearranging the terms, the Fourier transform of the gain term can be rewritten as
\begin{align}
\int_{\R} \int_{\R} \int_{\mathbb{S}^{d-1}} b_\sigma \left( \left\vert \frac{\xi}{\vert \xi \vert} \cdot \sigma \right\vert\right) \frac{\mathcal{F}(t,w)}{(2\pi)^d} e^{-i v\cdot \left[ -w + \frac{1-r}{2} \xi + \frac{1+r}{2} \vert \xi \vert \sigma \right]} \dd\sigma \dd w \dd v.
\end{align}
Consider now the following formula
\begin{align}\label{eq:inversionFormula}
\int_{\R} \int_{\R} \varphi(y) e^{-ix\cdot y} \dd y \dd x = \varphi(0),
\end{align}
which has to be understood in the sense $\displaystyle{\int_{\R} \widehat{\varphi}(x) \dd x = \int_{\R} \widehat{\varphi}(x) e^{i 0\cdot x} \dd x = (2\pi)^d \varphi(0)}$, where $\widehat{\varphi}$ is the Fourier transform of $\varphi$. Applying \eqref{eq:inversionFormula} to the function
\begin{align}
 \int_{\mathbb{S}^{d-1}} b_\sigma \left( \left\vert \frac{\xi}{\vert \xi \vert} \cdot \sigma \right\vert\right) \frac{\mathcal{F}(t, -w + \frac{1-r}{2} \xi + \frac{1+r}{2} \vert \xi \vert \sigma )}{(2\pi)^d} \dd\sigma,
\end{align}
we find in the end the following evolution equation for the Fourier transform $\mathcal{F}$ of the solution $f$ of \eqref{eq:LinInBolstrong_Cauchy}, namely
\begin{align}
\partial_t \mathcal{F}(t,\xi) = -i (a\cdot\xi) \mathcal{F}(t,\xi) + \int_{\mathbb{S}^{d-1}} b_\sigma \left( \left\vert \frac{\xi}{\vert \xi \vert} \cdot \sigma \right\vert\right) \mathcal{F}(t, \frac{1-r}{2} \xi + \frac{1+r}{2} \vert \xi \vert \sigma ) \dd\sigma - \mathcal{F}(t,\xi),
\end{align}
or rearranging the terms
\begin{align}
\label{EQUATFouriDifferentiVersi}
\partial_t \mathcal{F}(t,\xi) + \left( 1 + i(a\cdot\xi) \right) \mathcal{F}(t,\xi) = Q^+(\mathcal{F})(t, \xi),
\end{align}
where $\overline{\xi}$ is
\vspace{-7mm}
\begin{align}
\overline{\xi} = \frac{1-r}{2} \xi + \frac{1+r}{2} \vert \xi \vert \sigma.
\end{align}
and where $Q^+$ is the gain operator acting on $\mathcal{F}$
\begin{equation}\label{eq:gainQ^+}
     Q^+(\mathcal{F}) = \int_{\mathbb{S}^{d-1}} b_\sigma \left( \left\vert \frac{\xi}{\vert \xi \vert} \cdot \sigma \right\vert\right) \mathcal{F}(t, \overline{\xi} ) \dd\sigma
\end{equation}
\noindent
We remark that only very mild assumptions on the collision kernel $b$ are necessary to obtain that $\displaystyle{\int_{\mathbb{S}^{d-1}} b_\sigma \left( \left\vert \frac{\xi}{\vert \xi \vert} \cdot \sigma \right\vert\right) \mathcal{F}(t, \overline{\xi} ) \dd\sigma}$ is continuous provided that $\mathcal{F}(t,\cdot)\in \mathcal{C}([0,+\infty),\mathcal{C}_0(\R))$ . In fact, assuming that $\mathcal{F}$ is continuous, given $\xi, \xi_* \in \R$ we have that
\begin{align}
\left\vert \mathcal{F}(\xi) - \mathcal{F}(\xi_*) \right\vert = \left\vert \int_{\mathbb{S}^{d-1}} b_\sigma\left( \left\vert e_1 \cdot \sigma \right\vert \right) \left[ \mathcal{F}(\overline{\xi}_s) - \mathcal{F}(\overline{\xi_*}_c)\right] \dd \sigma \right\vert,
\end{align}
where $e_1$ denotes the first vector of the canonical basis and
\begin{align}
\overline{\xi}_R = \frac{1-r}{2} \xi + \frac{1+r}{2} \vert \xi \vert R_\xi(\sigma), \hspace{5mm} \text{and} \hspace{5mm} (\overline{\xi_*})_R = \frac{1-r}{2} \xi_* + \frac{1+r}{2} \vert \xi_* \vert R_{\xi_*}(\sigma),
\end{align}
with $R_{\xi}$ (respectively $R_{\xi_*}$) the orthogonal symmetry that sends $e_1$ to $\xi/\vert \xi \vert$ (respectively to $\xi_*/\vert \xi_* \vert$). The continuity follows from the fact that 
\begin{align}
\left\vert \left( \frac{1-r}{2} \xi + \frac{1+r}{2} \vert \xi \vert R_\xi(\sigma) \right) - \left( \frac{1-r}{2} \xi_* + \frac{1+r}{2} \vert \xi_* \vert R_{\xi_*}(\sigma) \right) \right\vert \longrightarrow 0
\end{align}
as $\vert \xi - \xi_* \vert \rightarrow 0$.

\end{appendices}

\bigskip

\noindent\textbf{Acknowledgements.} 
A.~Nota and T.~Dolmaire gratefully
acknowledge 
the support
by the project PRIN 2022 (Research Projects of National Relevance)
- Project code 202277WX43.\\
The authors are grateful to B.~Lods for useful discussions on the topic. 

\bigskip

\normalsize

\def\adresse{
\begin{description}

\item[T.~Dolmaire] { 
Dipartimento di Ingegneria e Scienze\\ dell'Informazione e Matematica (DISIM),\\ Università degli Studi dell'Aquila, \\ 67100  L'Aquila, Italy \\  E-mail: \texttt{theophile.dolmaire@univaq.it}} 

\item[N. Miele]{ Gran Sasso Science Institute,\\ Viale Francesco Crispi 7, 67100 L’Aquila, Italy  \\
E-mail: \texttt{nicola.miele@gssi.it}}

\item[A. Nota:] {Gran Sasso Science Institute,\\ Viale Francesco Crispi 7, 67100 L’Aquila, Italy  \\
E-mail: \texttt{alessia.nota@gssi.it}}

\end{description}
}

\adresse
\end{document}